\documentclass[12pt,a4paper]{amsart}
\usepackage{geometry}
\usepackage[colorlinks, linkcolor=red, anchorcolor=blue, citecolor=blue]{hyperref}
\usepackage{color}             
\usepackage{amsthm,bm} 
\usepackage{enumerate}
\usepackage{amsfonts}
\usepackage{amsmath}
\usepackage{mathrsfs}
\usepackage{amssymb}
\usepackage{amsbsy}
\usepackage{ifthen}
\usepackage[all]{xy}
\usepackage{extarrows}
\usepackage{indentfirst}        
\usepackage{latexsym}        

\usepackage{graphicx}
\usepackage{cases}
\usepackage{pifont}
\usepackage{txfonts}
\usepackage{xcolor}
\usepackage{multirow}
\usepackage{caption, subcaption}

\newcounter{maint}

\numberwithin{equation}{section}

\allowdisplaybreaks[4] 

\begin{document}

\newtheorem{theorem}{Theorem}[section]

\newtheorem{lemma}[theorem]{Lemma}

\newtheorem{corollary}[theorem]{Corollary}
\newtheorem{proposition}[theorem]{Proposition}

\theoremstyle{remark}
\newtheorem{remark}[theorem]{Remark}

\theoremstyle{definition}
\newtheorem{definition}[theorem]{Definition}

\theoremstyle{definition}
\newtheorem{conjecture}[theorem]{Conjecture}

\newtheorem{example}[theorem]{Example}
\newtheorem{problem}[theorem]{Problem}


\def\k{\Bbbk}
\def\I{\mathbb{I}}
\def\ufo{\mathfrak{ufo}}
\newcommand{\Dchaintwo}[4]{
\rule[-3\unitlength]{0pt}{8\unitlength}
\begin{picture}(14,5)(0,3)
\put(2,4){\ifthenelse{\equal{#1}{l}}{\circle*{4}}{\circle{4}}}
\put(4,4){\line(2,0){20}}
\put(26,4){\ifthenelse{\equal{#1}{r}}{\circle*{4}}{\circle{4}}}
\put(2,10){\makebox[0pt]{\scriptsize #2}}
\put(14,8){\makebox[0pt]{\scriptsize #3}}
\put(26,10){\makebox[0pt]{\scriptsize #4}}
\end{picture}}
\title[ Nichols algebras over  the Suzuki algebra 
\uppercase\expandafter{\romannumeral1}]{ Finite dimensional Nichols algebras over the Suzuki algebras  
\uppercase\expandafter{\romannumeral1}:  simple Yetter-Drinfeld modules of $A_{N\,2n}^{\mu\lambda}$}

\author[Shi]{Yuxing Shi }
\address{School of Mathematics and Statistics, Jiangxi Normal University,  Nanchang 330022, P. R. China}\email{yxshi@jxnu.edu.cn}

\subjclass[2010]{16T05, 16T25, 17B22}
\thanks{
\textit{Keywords:} Nichols algebra; Hopf algebra; Suzuki algebra; Yetter-Drinfeld module.
\\
This work was partially supported by
 Foundation of Jiangxi Educational Committee (No.12020447)
 and Natural Science Foundation of Jiangxi Normal University (No. 12018937).
}

\begin{abstract}
The Suzuki algebra $A_{Nn}^{\mu \lambda}$ was introduced by Suzuki Satoshi in 1998, which is 
a class of cosemisimple Hopf algebras. It is not categorically Morita-equivalent to a group algebra in general. In this paper, the author gives a complete set of simple Yetter-Drinfeld modules over the Suzuki algebra $A_{N\,2n}^{\mu\lambda}$  and investigates the Nichols algebras over those simple Yetter-Drinfeld modules. The involved finite dimensional Nichols algebras of 
diagonal type are of Cartan type $A_1$, $A_1\times A_1$, $A_2$, $A_2\times A_2$,
Super type 
${\bf A}_{2}(q;\I_2)$ and the Nichols algebra $\ufo(8)$. 
There are $64$, $4m$ and $m^2$-dimensional Nichols algebras of non-diagonal type
over $A_{N\,2n}^{\mu \lambda}$. 
The $64$-dimensional Nichols algebras are of dihedral rack type $\Bbb{D}_4$. 
The $4m$ and $m^2$-dimensional Nichols algebras $\mathfrak{B}(V_{abe})$  discovered first by Andruskiewitsch and Giraldi can 
be realized 
in the category of  Yetter-Drinfeld modules over $A_{Nn}^{\mu \lambda}$. By using a result of Masuoka, 
we prove that $\dim\mathfrak{B}(V_{abe})=\infty$ under the condition 
$b^2=(ae)^{-1}$,  $b\in\Bbb{G}_{m}$ for  $m\geq 5$.
\end{abstract}
\maketitle

\section{Introduction}

Let $\k$ be  an algebraicaly  closed field of characteristic $0$. The motivation of the paper 
is to make some contributions to the following project. 
\begin{problem}
How to classify all finite dimensional Hopf algebras over the Suzuki algebra 
$A_{Nn}^{\mu\lambda}$?
\end{problem}  
There are only a few  works to deal  with the problem. In 2004, Menini  and coauthors studied the quantum lines over $\mathcal{A}_{4m}$ and $\mathcal{B}_{4m}$, which are isomorphic to 
$A_{1m}^{++}$ and $A_{1m}^{+-}$ respectively \cite{MR2037722}. In 2019,  the author  \cite{Shi2019} classified finite dimensional Hopf algebras over the Kac-Paljutkin algebra $A_{12}^{+-} $
 and Fantino et al. \cite{Fantino2019} classified finite dimensional Hopf algebras over the dual of dihedral group $\k^{D_{2m}}$ of order $2m$, with  $m=4a\geq 12$, where $\k^{D_{2m}}$ is a $2$-cocycle deformation of 
$A_{1\,2a}^{++}$ \cite{MR1800713}. 

Why are we  interested in the Suzuki algebras? 
Firstly, our classification project is different with the classification of pointed Hopf algebras, 
since the Suzuki algebras are not categorically Morita-equivalent to  group algebras in general. 
The Suzuki algebras are non-trivial semisimple unless $(n,\lambda)=(2,+1)$. 
Two semisimple Hopf algebras $K$ and $H$ are categorically Morita-equivalent iff ${}_K^K\mathcal{YD}$ 
and ${}_H^H\mathcal{YD}$ are equivalent as braided tensor categories. 
In \cite{MR3677868}, Andruskiewitsch and coauthors constructed new Hopf algebras with semisimple Hopf algebras as coradicals 
 which are categorically Morita-equivalent to group algebras, including group-theoretical Hopf algebras, 
 in particular those from abelian extensions.  The Suzuki algebras can be obtained by 
 abelian extensions  \cite[Page 18]{Suzuki1998}, so they are group-theoretical \cite[Theorem 1.3]{Natale2003}.
But they are not categorically Morita-equivalent to group algebras in general, 
for example  $A_{12}^{+-}$ \cite[Section 5.2]{MR2386730}.  

Secondly, the study of Nichols algebras with non-group type braidings is rare and
it is an interesting problem to find new finite-dimensional Nichols algebras 
 over non-trivial semisimple Hopf algebras.
 In the past decades, the study of Nichols algebras are mainly focus on the Yetter-Drinfeld categories of group algebras.  In \cite[section 3.7]{Andruskiewitsch2018}, Andruskiewitsch
 and Giraldi found two classes of $4m$ and $m^2$-dimensional Nichols algebras (we call those Nichols algebras are of type $V_{abe}$) which generally 
 cannot be realized in the  Yetter-Drinfeld categories  of  group algebras. Together with the results of \cite{Shi2020odd}, we will see that those Nichols algebras  
 can be realized in the    Yetter-Drinfeld category of  $A_{Nn}^{\mu\lambda}$.
 
 Thirdly,  it is meaningful to provide many examples of Yetter-Drinfeld modules with 
 non-group type braidings. The classification of finite dimensional Nichols algebras of group type  has archived great success. For examples, Nichols algebras of diagonal type 
with finite dimension were classified completely  by Heckenberger \cite{heckenberger2009classification}
based on the theory of reflections  \cite{Heckenberger[2020]copyright2020}  and  Weyl groupoid \cite{MR2207786}; the classification of Nichols 
algebras of non-simple semisimple Yetter-Drinfeld modules over non-abelian groups were almost 
finished by  Heckenberger and  Vendramin \cite{Heckenberger2017} \cite{MR3605018}. 
To generalize the method of group type to non-group type, one obstacle is to realize  braidings 
of non-group type
in Yetter-Drinfeld categories.  In case of group type, 
Andruskiewitsch and Gra\~{n}a \cite{Andruskiewitsch2003MR1994219} built connections between  Nichols algebras of group-type and racks. It is convenient to realize braidings of rack type in 
Yetter-Drinfeld categories of finite groups. For this reason, the theory of reflections was also 
used to prove that the Nichols algebras of rack type $D$ is infinite dimensional \cite{Andruskiewitsch2011}.
Let $(V,c)$ be a rigid braided vector space.
It was shown by Schauenburg that $(V, c)$ can be realized on a coquasitriangular Hopf algebra 
$(H, \sigma)$ as a right $H$-comodule and $c$ arising from $\sigma$ \cite{Schauenburg1992} \cite{Takeuchi2000}. 
The realization is complicated 
in general. For example, the Suzuki algebras give a realization of $(V_{abe},c)$ \cite{Suzuki1998}.
As for the realization of non-group type braidings is not easy,  it is meaningful to provide examples 
for observation. 

Our classification project over the Suzuki algebras is based on the \textit{lifting method} 
introduced by Andruskiewitsch and Schneider \cite{MR1659895}. The lifting method is  a general framework  to classify finite dimensional non-semisimple Hopf algebras with a 
fixed sub-Hopf algebra  $H$ as coradical.  One crucial step of the lifting method 
is to find out all  Yetter-Drinfield module $V$ over $H$ such that the Nichols algebra $\mathfrak{B}(V)$ has finite dimension. 
In this paper, we 
deal with the Nichols algebras over simple Yetter-Drinfeld modules of $A_{Nn}^{\mu\lambda}$ with 
$n$ even. 
And in the sequel \cite{Shi2020odd}, we will study the case with $n$ odd. 

To investigate Nichols algebras, first we should know how to construct all 
Yetter-Drinfeld modules over a finite dimensional  Hopf algebra. Majid \cite{Majid1991}
identified the Yetter-Drinfeld modules 
with the modules of the Drinfeld double via the category equivalence 
${}_H^H\mathcal{YD}\simeq {}_{H^{cop}}\mathcal{YD}^{H^{cop}}
\simeq {}_{D(H^{cop})}\mathcal{M}$. 
And many mathematicians have contributed to the construction of Yetter-Drinfeld modules, for example
\cite{gould1993quantum}
\cite{Burciu2017}
\cite{Liu2019} \cite{radford2003oriented} \cite{MR2352888} \cite{MR2336009}.
We take Radford's method. 
There are exactly 
$8N^2$ one-dimensional, $2N^2(4n^2-1)$ two-dimensional and $8N^2$ $2n$-dimensional 
non-isomorphic Yetter-Drinfeld modules over $A_{N\,2n}^{\mu\lambda}$, see the Theorem
\ref{MainYDM}. 


The involved  Nichols algebras in the paper are of diagonal type, rack type,  type $V_{abe}$, 
and the Nichols algebra $\mathfrak{B}\left(\mathscr{K}_{jk,p}^{s}\right)$. 
The finite dimensional Nichols algebras of diagonal type over simple Yetter-Drinfeld modules of 
$A_{N\,2n}^{\mu\lambda}$ are of  Cartan type $A_1$, $A_1\times A_1$, $A_2$, $A_2\times A_2$,
 Super type 
${\bf A}_{2}(q;\I_2)$ and the Nichols algebra $\ufo(8)$. As a summary, we have the following  theorem. 
\begin{theorem}\label{MainTheorem}
Let $M$ be a simple Yetter-Drinfeld module over $A_{N\,2n}^{\mu\lambda}$. If 
$\mathfrak{B}(M)$ is of diagonal type and $\dim \mathfrak{B}(M)<\infty$, then $\mathfrak{B}(M)$
can be summarized  as follows. 
\begin{enumerate}
\item Cartan type $A_1$, see Lemmas \ref{Nichols_A} and \ref{Nichols_barA};
\item Cartan type $A_1\times A_1$, see Lemmas \ref{Nichols_B}, \ref{Nichols_C}, \ref{Nichols_D},
          \ref{Nichols_E}, \ref{Nichols_P}, \ref{Nichols_I}, \ref{Nichols_K}, \ref{Nichols_G}
           and \ref{Nichols_H};
\item Cartan type $A_2$, see Lemmas \ref{Nichols_B}, \ref{Nichols_C}, \ref{Nichols_D},
          \ref{Nichols_E}, \ref{Nichols_P}, \ref{Nichols_I}, \ref{Nichols_K}, \ref{Nichols_G}
           and \ref{Nichols_H};
\item  Cartan type $A_2\times A_2$, see Lemmas \ref{Nichols_I} and \ref{Nichols_K};      
\item Super type ${\bf A}_{2}(q;\I_2)$, see Lemmas \ref{Nichols_D} and \ref{Nichols_E};
\item The Nichols algebra $\ufo(8)$,  see Lemmas \ref{Nichols_D} and \ref{Nichols_E}.
\end{enumerate}
\end{theorem}

The Nichols algebra $\mathfrak{B}\left(\mathscr{I}_{pjk}^{s}\right)$ is of rack type.
If $n>2$, then it is of type $D$, see the Lemma \ref{Nichols_I_rack}.
If  $\mathfrak{B}\left(\mathscr{I}_{pjk}^{s}\right)$ is finite dimensional and $n=2$, then 
it is of dihedral rack type $\Bbb{D}_4$ or  Cartan type $A_2\times A_2$, see 
the Lemma \ref{Nichols_I}.

The Nichols algebras $\mathfrak{B}\left(\mathscr{G}_{jk,p}^{st}\right)$ and $\mathfrak{B}\left(\mathscr{H}_{jk,p}^{s}\right)$ are  of type $V_{abe}$, see the section \ref{typeVabe}. 
If $ae=b^2$, then $\mathfrak{B}(V_{abe})$ is of diagonal type. If $ae\neq b^2$, according to 
 \cite[section 3.7]{Andruskiewitsch2018} and the Corollary \ref{Vabe}, we have
\begin{align}\label{fomulaeVabe}
\dim\mathfrak{B}(V_{abe})
=\left\{\begin{array}{ll}
4m, &b=-1, ae\in\Bbb{G}_m,\\
m^2, &ae=1, b\in\Bbb{G}_m\,\,\text{for}\,\,m\geq 2,\\
\infty, & b^2=(ae)^{-1},  b\in\Bbb{G}_{m}\,\, \text{for}\,\, m\geq 5,\\
\infty,   & b\notin \Bbb{G}_m \,\,\text{for}\,\,m\geq 2,\\
\text{unknown}, & otherwise, 
\end{array}\right.
\end{align}
where $\Bbb{G}_m$ denotes the set of $m$-th primitive roots of unity. 

If  $\mathfrak{B}\left(\mathscr{K}_{jk,p}^{s}\right)$ is finite dimensional and $n=1$, 
then it is of Cartan type $A_1\times A_1$ or $A_2$. 
If  $\mathfrak{B}\left(\mathscr{K}_{jk,p}^{s}\right)$ is finite dimensional, $\lambda=1$,  and $n=2$, 
then it is of Cartan type $A_2\times A_2$. 
If $n>2$,  then the Nichols algebra $\mathfrak{B}\left(\mathscr{K}_{jk,p}^{s}\right)$
is complicated, see section \ref{NicholsKCasen3} for the case $n=3$.

The two unsolved cases in the paper  are difficult in general.  
\begin{problem}
Determine the dimensions of the following Nichols algebras. 
\begin{enumerate}
\item the unknown case in the formula \eqref{fomulaeVabe};
\item the unknown case for $\mathfrak{B}\left(\mathscr{K}_{jk,p}^{s}\right)$,  see the section \ref{section_Nichos_K}.
\end{enumerate}
\end{problem}

The paper is organized as follows. In the section 1, we introduce the motivation and background of the paper and summarize our main results. In the section 2, we make an introduction for the Suzuki algebra and
 construct all simple representations of $A_{N\,2n}^{\mu\lambda}$.
 In the section 3, we construct all simple Yetter-Drinfeld modules over $A_{N\,2n}^{\mu\lambda}$
 by using Radford's method and  put the construction of those Yetter-Drinfeld modules in the appendix.  In the section 4, we calculate Nichols algebras over simple Yetter-Drinfeld modules of $A_{N\,2n}^{\mu\lambda}$ in cases: diagonal type, type $V_{abe}$, the Nichols algebras 
$\mathfrak{B}\left(\mathscr{I}_{pjk}^{s}\right)$ and  $\mathfrak{B}\left(\mathscr{K}_{jk,p}^{s}\right)$. 
Finite dimensional Nichols algebras of diagonal type or non-diagonal type are obtained, and
there are two cases unsolved. 

\section{The Hopf algebra $A_{Nn}^{\mu\lambda}$
and the representations of $A_{N\,2n}^{\mu\lambda}$}
Suzuki introduced a family of cosemisimple Hopf algebras $A_{Nn}^{\mu\lambda}$ 
which is parametrized by integers $N\geq 1$, $n\geq 2$ and $\mu$, $\lambda=\pm 1$, 
and investigated various properties and structures of them \cite{Suzuki1998}. 
As explained in \cite[Example 13.3]{Takeuchi2002}, the suzuki algebras  give the braided vector space
$(V_{abe}, c)$ a realization in Yetter-Drinfeld category. 
Wakui studied the Suzuki algebra $A_{Nn}^{\mu\lambda}$ in perspectives of 
polynomial invariant \cite{Wakui2010a}, braided Morita invariant \cite{Wakui2019} and coribbon structures \cite{Wakui2003}. 
The Hopf algebra  $A_{Nn}^{\mu\lambda}$ is generated by $x_{11}$, $x_{12}$, $x_{21}$, 
$x_{22}$ subject to the relations:
\begin{align*}
&x_{11}^2=x_{22}^2,\quad x_{12}^2=x_{21}^2,\quad \chi _{21}^n=\lambda\chi _{12}^n,
\quad \chi _{11}^n=\chi _{22}^n,\\
&x_{11}^{2N}+\mu x_{12}^{2N}=1,\quad
x_{ij}x_{kl}=0\,\, \text{whenever $i+j+k+l$ is odd}, 
\end{align*}
where we use the following notation for $m\geq 1$, 
$$\chi _{11}^m:=\overbrace{x_{11}x_{22}x_{11}\ldots\ldots }^{\textrm{$m$ }},\quad \chi _{22}^m:=\overbrace{x_{22}x_{11}x_{22}\ldots\ldots }^{\textrm{$m$ }},$$
$$\chi _{12}^m:=\overbrace{x_{12}x_{21}x_{12}\ldots\ldots }^{\textrm{$m$ }},\quad \chi _{21}^m:=\overbrace{x_{21}x_{12}x_{21}\ldots\ldots }^{\textrm{$m$ }}.$$

The Hopf algebra structure of $A_{Nn}^{\mu\lambda}$ is given by
\begin{equation}\label{eq5.3}
\Delta (\chi_{ij}^k)=\chi_{i1}^k\otimes \chi_{1j}^k+\chi_{i2}^k\otimes \chi_{2j}^k,\quad 
\varepsilon(x_{ij})=\delta_{ij}, \quad S(x_{ij})=x_{ji}^{4N-1}, 
\end{equation}
for $k\geq 1$, $i,j=1,2$. 

Let $\overline{i,i+j}=\{i,i+1,i+2,\cdots,i+j\}$ be  an index set. 
Then the  basis of $A_{Nn}^{\mu\lambda}$ can be represented by 
\begin{equation}\label{eq5.2}
\left\{x_{11}^s\chi _{22}^t,\ x_{12}^s\chi _{21}^t \mid 
s\in\overline{1,2N}, t\in\overline{0,n-1} 
\right\}.
\end{equation}

Thus for $s,t\geq 0$ with $s+t\geq 1$, 
\begin{align*}
\Delta (x_{11}^s\chi _{22}^t)
&=x_{11}^s\chi _{22}^t\otimes x_{11}^s\chi _{22}^t
+x_{12}^s\chi _{21}^t\otimes x_{21}^s\chi _{12}^t,\\ 
\Delta (x_{12}^s\chi _{21}^t)
&=x_{11}^s\chi _{22}^t\otimes x_{12}^s\chi _{21}^t
+x_{12}^s\chi _{21}^t\otimes x_{22}^s\chi _{11}^t. 
\end{align*}

The cosemisimple Hopf algebra $A_{Nn}^{\mu\lambda}$ is decomposed to the direct sum of simple subcoalgebras such as $A_{Nn}^{\mu\lambda}
=\bigoplus_{g\in G}\k g\oplus\bigoplus_{\substack{0\leq s\leq N-1\\ 1\leq t\leq n-1}}C_{st}$
\cite[Theorem 3.1]{Suzuki1998}\cite[lemma 5.5]{Wakui2010a}, 
where 
\begin{align*}
G&=\left\{x_{11}^{2s}\pm x_{12}^{2s}, x_{11}^{2s+1}\chi_{22}^{n-1}\pm 
        \sqrt{\lambda}x_{12}^{2s+1}\chi_{21}^{n-1}\mid s\in\overline{1,N}\right\},\\
C_{st}&=\k x_{11}^{2s}\chi_{11}^t+\k x_{12}^{2s}\chi_{12}^t+
              \k x_{11}^{2s}\chi_{22}^t+\k x_{12}^{2s}\chi_{21}^t,\quad
              s\in\overline{1,N}, t\in\overline{1,n-1}. 
\end{align*}
The set $\left\{\k g\mid g\in G\right\}\cup \left\{\k x_{11}^{2s}\chi_{11}^t
+\k x_{12}^{2s}\chi_{21}^t \mid s\in\overline{1,N}, 
t\in\overline{1,n-1}
\right\}$ 
is a full set of non-isomorphic simple left $A_{Nn}^{\mu\lambda}$-comodules, where 
the coactions of the comodules listed above are given by the coproduct $\Delta$. Denote 
the comodule $\k x_{11}^{2s}\chi_{11}^t
+\k x_{12}^{2s}\chi_{21}^t $ by $\Lambda_{st}$. That is to say the comodule 
$\Lambda_{st}=\k w_1+\k w_2$
is defined as 
\begin{align*}
\rho\left(w_1\right)
=  x_{11}^{2s}\chi_{11}^t\otimes w_1
     +x_{12}^{2s}\chi_{12}^t\otimes w_2,\quad 
\rho\left(w_2\right)
=  x_{11}^{2s}\chi_{22}^t\otimes w_2
     +x_{12}^{2s}\chi_{21}^t\otimes w_1 .
\end{align*}

\begin{proposition}
Let  $\omega$ be a primitive $8nN$-th root of unity. 
Set 
\[
\tilde{\mu}=\left\{\begin{array}{rl}
                    1,&\mu=1\\
                    \omega^{2n}, &\mu=-1,
                    \end{array}\right.
\qquad
\bar{\mu}=\left\{\begin{array}{rl}
                   1, & \mu=1,\vspace{1mm}\\
                   \omega^{4n}, & \mu=-1.
                   \end{array}\right.
\]
Then a full set of non-isomorphic simple left $A_{N\,2n}^{\mu\lambda}$-modules is given by 
\begin{enumerate}
\item $V_{ijk}=\k v$, $i,j\in\Bbb{Z}_2$, $k\in\overline{0,N-1}$. The action of $A_{N\,2n}^{\mu\lambda}$ on $V_{ijk}$ is given by 
\[
x_{12}\mapsto 0,\quad x_{21}\mapsto 0,\quad 
 x_{11}\mapsto (-1)^i\omega^{4nk},\quad  x_{22}\mapsto (-1)^j\omega^{4nk};
 \]
\item $V_{ijk}^\prime=\k v$, $\lambda=1$, $i,j\in\Bbb{Z}_2$, $k\in\overline{0,N-1}$.  The action of $A_{N\,2n}^{\mu+}$ on $V_{ijk}^\prime$ is given by 
\[
x_{11}\mapsto 0,\quad x_{22}\mapsto 0,\quad 
 x_{12}\mapsto (-1)^i\omega^{4nk}\tilde{\mu},\quad  x_{21}\mapsto (-1)^j\omega^{4nk}\tilde{\mu};
 \]

\item  $V_{jk}=\k v_{1}\oplus \k v_{2}$, $k\in\overline{0,N-1}$, $\frac j2\in\overline{0,n-1}$. The action of $A_{N\,2n}^{\mu\lambda}$ on the row vector $(v_{1}, v_{2})$ 
  is given by 
 \[
 x_{11}\mapsto \begin{pmatrix}0&\omega^{2(4kn-jN)}\\
 \omega^{2jN}&0
 \end{pmatrix},\quad x_{12},x_{21}\mapsto 0,\quad 
  x_{22}\mapsto \begin{pmatrix}0&\omega^{8kn}\\
 1&0
 \end{pmatrix};
 \]
 
\item $V_{jk}^\prime=\k v_{1}^\prime\oplus \k v_{2}^\prime$, $k\in\overline{0,N-1}$, 
$\left\{\begin{array}{ll}\frac j2\in\overline{1,n-1}, &\lambda=1,\vspace{1mm}\\
\frac {j+1}2\in\overline{1,n}, &\lambda=-1.
\end{array}\right.$  The action of $A_{N\,2n}^{\mu\lambda}$ on the row vector $(v_{1}^\prime, v_{2}^\prime)$ 
  is given by 
 \[
 x_{21}\mapsto \begin{pmatrix}0& \bar{\mu}\omega^{2(4kn-jN)}\\
 \omega^{2jN}&0
 \end{pmatrix},\quad x_{11},x_{22}\mapsto 0,\quad 
  x_{12}\mapsto \begin{pmatrix}0&\bar{\mu}\omega^{8kn}\\
 1&0
 \end{pmatrix}.
 \]
\end{enumerate}
\end{proposition}
We leave the proof to the reader since it's easy and tedious.

\section{Yetter-Drinfeld modules over $A_{N\,2n}^{\mu\lambda}$}
Similarly according to Radford's method \cite[Proposition 2]{radford2003oriented}, any simple left Yetter-Drinfeld module over  a Hopf algebra $H$ could be constructed by the submodule of  tensor product of a left  module $V$ of $H$ and
$H$ itself, where the  module and comodule structures are given by :
\begin{align}
h\cdot (\ell\boxtimes g)
&=(h_{(2)}\cdot \ell)\boxtimes h_{(1)}gS(h_{(3)}),\label{eq:action}\\
\rho(\ell\boxtimes h)
&=h_{(1)}\otimes (\ell\boxtimes h_{(2)}) , \forall h, g\in H, \ell\in V.
\label{eq:coaction}
\end{align}
Here we use $\boxtimes$ instead of $\otimes$ to avoid confusion by using too many symbols of the tensor product.  we construct all simple left Yetter-Drinfeld modules over $A_{N\,2n}^{\mu\lambda}$ in this way and put them in the appendix. 

Let $V$ be a simple left $A_{N\,2n}^{\mu\lambda}$ module, we decompose 
$V\boxtimes A_{N\,2n}^{\mu\lambda}$ into small Yetter-Drinfeld modules.
The left $A_{N\,2n}^{\mu\lambda}$-module structure of 
$V_{ijk}\boxtimes A_{N\,2n}^{\mu\lambda}$ is decided  by formulas in 
the Figure \ref{figureFormulae}.

\begin{figure}
\begin{align*}
x_{pq}\cdot \left(v\boxtimes x_{11}^s\chi_{22}^t\right)
&=\left\{\begin{array}{ll}
(-1)^i\omega^{4nk}v\boxtimes x_{11}^s, &pq=11, t=0, \vspace{1mm}\\
(-1)^i\omega^{4nk}v\boxtimes x_{11}^{s+1}\chi_{22}^{t-1}, 
        &pq=11, t\,\,\text{even}, t>0,\vspace{1mm}\\
(-1)^i\omega^{4nk}v\boxtimes x_{11}^{s-1}\chi_{22}^{t+1}, &pq=11, t\,\,\text{odd}, \vspace{1mm}\\
(-1)^j\omega^{4nk}v\boxtimes x_{11}^{s-1}\chi_{22}^{t+1}, 
        &pq=22, s\,\,\text{odd}, t\,\,\text{odd}, \vspace{1mm}\\
(-1)^j\omega^{4nk}v\boxtimes x_{11}^{s-3}\chi_{22}^{t+3}, 
        &pq=22, s\,\,\text{odd}, t\,\,\text{even}, \vspace{1mm}\\
(-1)^j\omega^{4nk}v\boxtimes x_{11}^{s}, 
        &pq=22, s\,\,\text{even}, t=0, \vspace{1mm}\\
(-1)^j\omega^{4nk}v\boxtimes x_{11}^{s}\chi_{22},  
        &pq=22, s\,\,\text{even}, t=1, \vspace{1mm}\\
(-1)^j\omega^{4nk}v\boxtimes x_{11}^{s+1}\chi_{22}^{t-1}, 
        &pq=22, s\,\,\text{even}, 0<t\,\,\text{even}, \vspace{1mm}\\
(-1)^j\omega^{4nk}v\boxtimes x_{11}^{s+3}\chi_{22}^{t-3}, 
        &pq=22, s\,\,\text{even},  1<t\,\,\text{odd}, \vspace{1mm}\\
0, &otherwise,
\end{array}\right.\\
x_{pq}\cdot \left(v\boxtimes x_{12}^s\chi_{21}^t\right)
&=\left\{\begin{array}{ll}
(-1)^j\omega^{4nk}v\boxtimes x_{12}^s, &pq=11, t=0, \vspace{1mm}\\
(-1)^j\omega^{4nk}v\boxtimes x_{12}^{s+1}\chi_{21}^{t-1}, 
        &pq=11, t\,\,\text{even}, t>0,\vspace{1mm}\\
(-1)^j\omega^{4nk}v\boxtimes x_{12}^{s-1}\chi_{21}^{t+1}, &pq=11, t\,\,\text{odd}, \vspace{1mm}\\
(-1)^i\omega^{4nk}v\boxtimes x_{12}^{s-1}\chi_{21}^{t+1}, 
        &pq=22, s\,\,\text{odd}, t\,\,\text{odd}, \vspace{1mm}\\
(-1)^i\omega^{4nk}v\boxtimes x_{12}^{s-3}\chi_{21}^{t+3}, 
        &pq=22, s\,\,\text{odd}, t\,\,\text{even}, \vspace{1mm}\\
(-1)^i\omega^{4nk}v\boxtimes x_{12}^{s}, 
        &pq=22, s\,\,\text{even}, t=0, \vspace{1mm}\\
(-1)^i\omega^{4nk}v\boxtimes x_{12}^{s}\chi_{21},  
        &pq=22, s\,\,\text{even}, t=1, \vspace{1mm}\\
(-1)^i\omega^{4nk}v\boxtimes x_{12}^{s+1}\chi_{21}^{t-1}, 
        &pq=22, s\,\,\text{even}, 0<t\,\,\text{even},  \vspace{1mm}\\
(-1)^i\omega^{4nk}v\boxtimes x_{12}^{s+3}\chi_{21}^{t-3}, 
        &pq=22, s\,\,\text{even}, 1<t\,\,\text{odd}, \vspace{1mm}\\
0, &otherwise.
\end{array}\right.
\end{align*}
\caption{The action of generators on $V_{ijk}\boxtimes A_{N\,2n}^{\mu\lambda}$}
\label{figureFormulae}
\end{figure}

We can decompose $V_{ijk}\boxtimes A_{N\,2n}^{\mu\lambda}$ into small Yetter-Drinfeld modules as 
\begin{align*}
V_{ijk}\boxtimes A_{N\,2n}^{\mu\lambda}
\simeq \bigoplus_{s=1}^N
\left[M_{ijk}^s\oplus N_{ijk}^s\oplus \bigoplus_{t=-1}^{n-1}(V_{ijk}\boxtimes C_{s\, 2t+2})\right],
\end{align*}
where $C_{s\, 0}:=\k x_{11}^{2s}+\k x_{12}^{2s}$, 
$C_{s\, 2n}:=\k x_{11}^{2s}\chi_{11}^{2n}+\k x_{12}^{2s}\chi_{12}^{2n}$, and 
\[
V_{ijk}\boxtimes C_{s\, 2t+2}\simeq \left\{\begin{array}{ll}
\mathscr{C}_{ijk,0}^{st}\oplus \mathscr{C}_{ijk,1}^{st},& t\in\overline{0,n-2},\\
\mathscr{B}_{ijk}^{s},& t=-1, j=i+1,\\
\mathscr{A}_{iik,0}^{s}\oplus \mathscr{A}_{iik,1}^{s},& t=-1, j=i,\\
\mathscr{C}_{ijk,0}^{st}, & t=n-1, \text{see Table}\,\,\ref{YDMod1}, \vspace{1mm}\\
\bar{\mathscr{A}}_{ijk,0}^{s}\oplus \bar{\mathscr{A}}_{ijk,1}^{s},& t=n-1, 
i=j(\text{or}\,j+1), \lambda=1(\text{or}\,-1).
\end{array}\right.
\]

Similarly, we have 
\[
V_{jk}\boxtimes A_{N\,2n}^{\mu\lambda}\simeq 
\bigoplus_{s=1}^N\left[
\bigoplus_{p=0}^1\left(\mathscr{I}_{pjk}^s\oplus \mathscr{J}_{pjk}^s\right)\oplus
\bigoplus_{t=-1}^{n-1}(V_{jk}\boxtimes C_{s\, 2t+2})
\right], 
\]
where $V_{jk}\boxtimes C_{s\, 2t+2}
\simeq \left\{\begin{array}{ll}
\bigoplus_{p=0}^1\left(\mathscr{D}_{jk,p}^{st}\oplus\mathscr{E}_{jk,p}^{s\,t+1}\right), 
& 0\leq t\leq n-2,\vspace{1mm}\\
\mathscr{D}_{jk,0}^{st}\oplus \mathscr{D}_{jk,1}^{st},&t=n-1,\vspace{1mm}\\
\mathscr{E}_{jk,0}^{s0}\oplus \mathscr{E}_{jk,1}^{s0},&t=-1;
\end{array}\right.$
\[
V_{jk}^\prime\boxtimes A_{N\,2n}^{\mu\lambda}\simeq 
\bigoplus_{s=1}^N\left[
\bigoplus_{p=0}^1\left(\mathscr{K}_{jk,p}^s\oplus \mathscr{L}_{jk,p}^s\right)
\oplus\bigoplus_{t=0}^{n-1}(V_{jk}^\prime\boxtimes C_{s\, 2t+1})
\right], 
\]
where $V_{jk}^\prime\boxtimes C_{s\, 2t+1}\simeq \bigoplus_{p=0}^1
\left(\mathscr{G}_{jk,p}^{st}\oplus \mathscr{H}_{jk,p}^{st}\right)$;
\[
V_{ijk}^\prime\boxtimes A_{N\,2n}^{\mu\lambda}
\simeq \bigoplus_{s=1}^N\left[\mathscr{L}_{ijk,0}^s\oplus \mathscr{L}_{ijk,1}^s
\oplus\bigoplus_{t=0}^{n-1}(V_{ijk}^\prime\boxtimes C_{s\,2t+1})
\right],
\]
where $V_{ijk}^\prime\boxtimes C_{s\,2t+1}\simeq \bigoplus_{p=0}^1\mathscr{P}_{ijk,p}^{st}$.

Our strategy is to break $V\boxtimes A_{N\,2n}^{\mu\lambda}$ into 
small sub-Yetter-Drinfeld modules which can't  break any more, and single out a complete set 
of simple 
Yetter-Drinfeld modules over $A_{N\,2n}^{\mu\lambda}$ from those submodules.
According to the above decompositions, we see the dimension 
distribution of those small Yetter-Drinfeld 
modules is $1$, $2$ and $2n$. From the appendix and table \ref{YDMod1}, 
it is not difficult to see that there are $8N^2$ pairwise non-isomorphic Yetter-Drinfeld modules of  dimension one and  $2N^2(4n^2-1)$ pairwise non-isomorphic Yetter-Drinfeld modules of  
dimension two, see the Theorem \ref{MainYDM}.  There are  seven classes of Yetter-Drinfeld modules of  dimension $2n$
 in total and they have the following relations. 
\begin{enumerate}
\item $\mathscr{I}_{pj k}^s\simeq \mathscr{I}_{pj^\prime k}^s$
          in case of $j\equiv j^\prime \mod 4$.
\item If $n$ is even, then $\mathscr{M}_{ijk}^s\simeq \mathscr{I}_{pj^\prime k}^s$
          in case of $i=p$, $j^\prime \mid 4$.
\item  If $n$ is odd, then $\mathscr{M}_{ijk}^s\simeq \mathscr{I}_{pj^\prime k}^s$
          in case of $i=p$ and \\
         $
        j^\prime \equiv
        \left\{\begin{array}{ll}
        0 \mod 4, &\text{if}\, i+j\,\,\,\text{is even},\\
        2 \mod 4, &\text{if}\, i+j\,\,\,\text{is odd}.
        \end{array}\right.$
\item If $n$ is even, then $\mathscr{N}_{ijk}^s\simeq \mathscr{I}_{pj^\prime k}^s$
         in case of $j=p$ and \\
         $
        j^\prime \equiv
        \left\{\begin{array}{ll}
        0 \mod 4, &\text{if}\,\lambda=1,\\
        2 \mod 4, &\text{if}\,\lambda=-1.
        \end{array}\right.$
\item If $n$ is odd, then $\mathscr{N}_{ijk}^s\simeq \mathscr{I}_{pj^\prime k}^s$
         in case of $j=p$ and \\
         $
        j^\prime \equiv
        \left\{\begin{array}{ll}
        0 \mod 4, &\text{if}\,\lambda=1, i+j\,\,\,\text{is even},\\
        2 \mod 4, &\text{if}\,\lambda=1, i+j\,\,\,\text{is odd},\\
        2 \mod 4, &\text{if}\,\lambda=-1, i+j\,\,\,\text{is even},\\
        0 \mod 4, &\text{if}\,\lambda=-1, i+j\,\,\,\text{is odd}.
        \end{array}\right.$
\item $\mathscr{I}_{pjk}^s\simeq \mathscr{J}_{pj^\prime k}^s$ in case of 
         $j\equiv 
         \left\{\begin{array}{ll}
         j^\prime \mod 4, &\text{if}\, \lambda=1,\\
         j^\prime +2\mod 4, &\text{if}\, \lambda=-1.
         \end{array}\right.$
\item $\mathscr{K}_{jk,p}^s\simeq \mathscr{K}_{j^\prime k,p}^s$ in case of $j
          \equiv j^\prime\,\,\mathrm{mod}\,\, 4$.
\item $\mathscr{K}_{jk,p}^s\simeq \mathscr{L}_{j^\prime k,p}^s$ in case of $j+j^\prime
          \equiv 0\,\,\mathrm{mod}\,\, 4$.
\item $\mathscr{Q}_{ijk,p}^s\simeq \mathscr{L}_{j^\prime k,p}^s$ in case of
$$
j^\prime\equiv \left\{\begin{array}{ll}
0 \mod 4,&\text{if}\,i+j\,\,\text{is even and}\,\, n\,\text{is odd},\\
2 \mod 4,&\text{if}\,i+j\,\,\text{is odd and}\,\, n\,\text{is odd},\\
0\mod 4, &\text{if}\,n\,\text{is even}.
\end{array}\right.
$$
\end{enumerate}
Since $8N^2\cdot 1^2+2N^2(4n^2-1)\cdot 2^2+8N^2\cdot (2n)^2=(8Nn)^2$, all the simple Yetter-Drinfeld modules are given by  the following theorem. 

\begin{theorem}\label{MainYDM}
A full set of  non-isomorphic simple Yetter-Drinfeld modules over $A_{N\,2n}^{\mu\lambda}$ is given by the following list.
\begin{enumerate}
\item There are $8N^2$ non-isomorphic Yetter-Drinfeld modules of  dimention one:
\begin{enumerate}
\item $\mathscr{A}_{iik,p}^s$, $i, p\in \Bbb{Z}_2$,  $s\in\overline{1,N}$, $k\in\overline{0,N-1}$;
\item $\bar{\mathscr{A}}_{i\,i+1\,k,p}^s$, $i, p\in \Bbb{Z}_2$,  $s\in\overline{1,N}$, $k\in\overline{0,N-1}$.
\end{enumerate}
\item There are $2N^2(4n^2-1)$  non-isomorphic Yetter-Drinfeld modules of 
dimension two:
\begin{enumerate}
\item $\mathscr{B}_{01k}^s$, $s\in\overline{1,N}$, $k\in\overline{0,N-1}$;
\item $\mathscr{C}_{ijk,p}^{st}$, $ij=00$ or $01$, $k\in\overline{0,N-1}$, $p\in \Bbb{Z}_2$,  
          $s\in\overline{1,N}$, $t\in\overline{0,n-2}$;
\item $\mathscr{C}_{ijk,p}^{st}$, $i=0$, 
          $j=\left\{\begin{array}{ll} i+1,&\text{if}\,\lambda=1,\\i,&\text{if}\,\lambda=-1,\end{array}\right.$
          $k\in\overline{0,N-1}$, $s\in\overline{1,N}$, $p=0$,
          $t=n-1$;
\item $\mathscr{D}_{jk,p}^{st}$, $\frac j2\in\overline{1,n-1}$, 
          $k\in\overline{0,N-1}$, $p\in \Bbb{Z}_2$, $s\in\overline{1, N}$, $t\in\overline{0,n-1}$; 
\item $\mathscr{E}_{jk,p}^{st}$, $\frac j2\in\overline{1,n-1}$, 
          $k\in\overline{0,N-1}$, $p\in \Bbb{Z}_2$, $s\in\overline{1,N}$, $t\in\overline{0,n-1}$; 
\item $\mathscr{G}_{jk,p}^{st}$, 
          $\left\{\begin{array}{ll}\frac j2\in\overline{1,n-1}, &\text{if}\,\lambda=1,\vspace{1mm}\\
             \frac {j+1}2\in\overline{1,n}, &\text{if}\,\lambda=-1,
          \end{array}\right.$
          $k\in\overline{0,N-1}$, $p\in \Bbb{Z}_2$, $s\in\overline{1,N}$, 
          $t\in\overline{0,n-1}$; 
\item $\mathscr{H}_{jk,p}^{st}$, 
          $\left\{\begin{array}{ll}\frac j2\in\overline{1,n-1}, &\text{if}\,\lambda=1,\vspace{1mm}\\
             \frac {j+1}2\in\overline{1,n}, &\text{if}\,\lambda=-1,
          \end{array}\right.$
          $k\in\overline{0,N-1}$, $p\in \Bbb{Z}_2$, $s\in\overline{1,N}$, 
          $t\in\overline{0,n-1}$;
\item $\mathscr{P}_{ijk,p}^{st}$, $ij=00$ or $01$, $k\in\overline{0,N-1}$, $p\in \Bbb{Z}_2$,    
          $s\in\overline{1,N}$, $t\in\overline{0,n-1}$.
\end{enumerate} 
\item There are $8N^2$  non-isomorphic Yetter-Drinfeld modules of 
dimension $2n$:
\begin{enumerate}
\item $\mathscr{I}_{pjk}^s$, $j=2$ or $4$, 
          $k\in\overline{0,N-1}$, $p\in \Bbb{Z}_2$,
          $s\in\overline{1,N}$;
\item $\mathscr{K}_{jk,p}^s$, 
          $j=\left\{\begin{array}{ll}
          1\,\text{or}\,3, &\text{if}\,\lambda=-1,\\
          2\,\text{or}\,4, &\text{if}\,\lambda=1,
          \end{array}\right.
          $
          $k\in\overline{0,N-1}$, $p\in \Bbb{Z}_2$,
          $s\in\overline{1,N}$.
\end{enumerate}
\end{enumerate}
\end{theorem}
\begin{remark}
As for the description of those Yetter-Drinfeld modules, please see the Appendix. 
To simplify the notations, we alway allow $j$ takes values $2$ and $4$ for $\mathscr{I}_{pjk}^s$, 
and the similar settings for $\mathscr{K}_{jk,p}^s$. 
\end{remark}

\begin{table}
\newcommand{\tabincell}[2]{\begin{tabular}{@{}#1@{}}#2\end{tabular}}
\begin{tabular}{|c|c|c|c|c|}
\hline
YD-mod 
& dim &  parameters &  mod   & comod\\\hline
$\mathscr{B}_{ijk}^s$ 
& $2$
&\tabincell{c}{\vspace{1mm} $i,j\in\Bbb{Z}_2$, $j=i+1$,
\\$s\in\overline{1,N}$,  $k\in\overline{0,N-1}$\vspace{1mm}} 
&  $V_{ijk}\oplus V_{jik}$  
&\tabincell{c}{$\k g^+_s \oplus \k g^-_s $}
\\\hline
\tabincell{c}{$\mathscr{C}_{ijk,p}^{st}$}
& $2$
& \tabincell{c}{$i,j,p\in\Bbb{Z}_2$,\\$s\in\overline{1,N}$,  $t\in\overline{0,n-2}$,\\
    $k\in\overline{0,N-1}$}
& $V_{ijk}\oplus V_{i+1\,j+1\,k}$ 
& $\Lambda_{s\,2t+2}$
\\\hline
\tabincell{c}{$\mathscr{C}_{ijk,p}^{st}$}
& $2$
& \tabincell{c}{$i,j,p\in\Bbb{Z}_2$,\\
    $j=\left\{\begin{array}{ll} i+1,&\lambda=1,\\i,&\lambda=-1,\end{array}\right.$\\
    $s\in\overline{1,N}$,  $t=n-1$,\\
    $k\in\overline{0,N-1}$}
& $V_{ijk}\oplus V_{i+1\,j+1\,k}$ 
& $\k h^+_s \oplus \k h^-_s$
\\\hline
\multirow{2}{*}{$\mathscr{D}_{jk,p}^{st}$}
& \multirow{2}{*}{$2$}
&\multirow{2}{*}{ \tabincell{c}{$s\in\overline{1,N}$,  $t\in\overline{0,n-1}$,\\
    $\frac j2\in\overline{1,n-1}$, $p\in\Bbb{Z}_2$,\\$k\in\overline{0,N-1}$}}
& \multirow{2}{*}{$V_{jk}$} 
&  \tabincell{c}{ $t\neq n-1$, \\$\Lambda_{s\,2t+2}$}
\\\cline{5-5}
&&&& \tabincell{c}{ $t=n-1$, \\$\k h^+_s \oplus \k h^-_s $}
\\\hline
\multirow{2}{*}{$\mathscr{E}_{jk,p}^{st}$}
& \multirow{2}{*}{$2$}
& \multirow{2}{*}{
\tabincell{c}{$s\in\overline{1,N}$,  $t\in\overline{0,n-1}$,\\
    $\frac j2\in\overline{1,n-1}$, $p\in\Bbb{Z}_2$,\\$k\in\overline{0,N-1}$}
}
& \multirow{2}{*}{$V_{jk}$} 
& \tabincell{c}{ $t\neq 0$, \\$\Lambda_{s\, 2t}$}
\\\cline{5-5}
&&&& \tabincell{c}{ $t=0$, \\$\k g^+_s \oplus \k g^-_s $}
\\\hline
$\mathscr{G}_{jk,p}^{st}$
& $2$
&\tabincell{c}{
$\left\{\begin{array}{ll}\frac j2\in\overline{1,n-1}, &\lambda=1,\vspace{1mm}\\
\frac {j+1}2\in\overline{1,n}, &\lambda=-1,
\end{array}\right.$\\
$s\in\overline{1,N}$, $t\in\overline{0,n-1}$,\\ $k\in\overline{0,N-1}$, $p\in\Bbb{Z}_2$ }
&  $V_{jk}^\prime$
& $\Lambda_{s\,2t+1}$
\\\hline
$\mathscr{H}_{jk,p}^{st}$
& $2$
&\tabincell{c}{
$\left\{\begin{array}{ll}\frac j2\in\overline{1,n-1}, &\lambda=1,\vspace{1mm}\\
\frac {j+1}2\in\overline{1,n}, &\lambda=-1,
\end{array}\right.$\\
$s\in\overline{1,N}$, $t\in\overline{0,n-1}$,\\ $k\in\overline{0,N-1}$, $p\in\Bbb{Z}_2$ }
&  $V_{jk}^\prime$
& $\Lambda_{s\,2t+1}$
\\\hline
\tabincell{c}{$\mathscr{P}_{ijk,p}^{st}$}
& $2$
& \tabincell{c}{$\lambda=1$, $i,j,p\in\Bbb{Z}_2$,\\$s\in\overline{1,N}$,  $t\in\overline{0,n-1}$,\\
    $k\in\overline{0,N-1}$}
& $V_{ijk}^\prime\oplus V_{i+1\,j+1\,k}^\prime$ 
& $\Lambda_{s\,2t+1}$
\\\hline
\end{tabular}
\vspace{3ex}
\caption{Two dimensional simple Yetter-Drinfeld modules over $A_{N\,2n}^{\mu\lambda}$ . Here $g^{\pm}_s=x_{11}^{2s}\pm x_{12}^{2s}$, $h^{\pm}_s=x_{11}^{2s}\chi_{11}^{2n}
\pm\sqrt{\lambda} x_{12}^{2s}\chi_{12}^{2n}$.}
\label{YDMod1}
\end{table}

\section{Nichols algebras over simple Yetter-Drinfeld modules}
In this section, we investigate  Nichols algebras over  simple Yetter-Drinfeld modules
of $A_{N\,2n}^{\mu\lambda}$. 
So the Yetter-Drinfeld modules discussed in the section are those listed in Theorem \ref{MainYDM}.
For the knowledge about Nichols algebras, please refer to 
\cite{Heckenberger[2020]copyright2020} \cite{Andruskiewitsch2017}.
\subsection{Nichols algebras of diagonal type}
Let $V=\bigoplus_{i\in I}\k v_i$ be a vector space with a braiding $c(v_i\otimes v_j)=q_{ij}v_j\otimes v_i$, 
$q_{ij}\in\k^\times$, then the Nichols algebra $\mathfrak{B}(V)$ is  of diagonal type. 
Our results in this section heavily rely on Heckenberger's classification work  \cite{heckenberger2009classification}. To keep the article concise, we do not repeat this in the following proofs.  
As for the Nichols algebra $\ufo(8)$, we mean the Nichols algebra with the Dynkin diagram 
   $\xymatrix{ \overset{-\zeta^2}{\underset{\ }{\circ}}\ar  @{-}[rr]^{\zeta}  &&
\overset{-\zeta^2}{\underset{\ }{\circ}}}$,  $\zeta\in\mathbb{G}_{12}$, 
see \cite[Page 561]{Andruskiewitsch2017}. 

\begin{lemma}\label{Nichols_A}
$
\dim \mathfrak{B}\left(\mathscr{A}_{iik,p}^s\right)
=\left\{\begin{array}{ll}
\infty, & N\mid ks,\\
\frac{N}{(N,\, ks)}, &N\nmid ks.
\end{array}\right.
$
\end{lemma}
\begin{proof}
$c(w\otimes w)
=\left[(-1)^i\omega^{4kn}\right]^{2s}w\otimes w
=\omega^{8nks}w\otimes w$.
\end{proof}
\begin{corollary}
\begin{enumerate}
\item If $N=1$, then 
$
\dim \mathfrak{B}\left(\mathscr{A}_{iik,p}^s\right)=\infty.
$
\item If $N$ is a prime, then 
$\dim \mathfrak{B}\left(\mathscr{A}_{iik,p}^s\right)
=\left\{\begin{array}{ll}
\infty, & k=0\,\text{or}\, s=N, \\
N, &otherwise.
\end{array}\right.$
\end{enumerate}
\end{corollary}

\begin{lemma}\label{Nichols_barA}
\[
\dim \mathfrak{B}\left(\bar{\mathscr{A}}_{ijk,p}^s\right)
=\left\{\begin{array}{ll}
\infty, & \lambda=1,  N\mid k(s+n), \\
\frac{N}{(N,\, k(s+n))}, &\lambda=1,  N\nmid k(s+n), \\
\infty, & \lambda=-1,  2N\mid [Nn+2k(s+n)], \\
\frac{2N}{(2N,\, Nn+2k(s+n))}, &\lambda=-1,  2N\nmid [Nn+2k(s+n)].
\end{array}\right.
\]
\end{lemma}
\begin{proof}
$
c(w\otimes w)
=(-1)^{(i+j)n}\omega^{8kn(s+n)}w\otimes w.
$
\end{proof}
\begin{corollary}
If  $N=1$, then 
$\dim \mathfrak{B}\left(\bar{\mathscr{A}}_{ijk,p}^s\right)
=\left\{\begin{array}{ll}
2, &\lambda=-1,\, n\,\text{odd},\\
\infty, &otherwise.
\end{array}\right.
$
\end{corollary}

\begin{lemma}\label{Nichols_B}
$
\dim \mathfrak{B}\left(\mathscr{B}_{ijk}^s\right)
=\left\{\begin{array}{ll} 
4, &N\mid 2ks, N\nmid ks, \text{(Cartan type $A_1\times A_1$)},\\ 
27, &N\mid 3ks, N\nmid ks, 
\text{(Cartan type $A_2$)},\\
\infty, & otherwise. 
\end{array}\right.
$
\end{lemma}
\begin{proof}
$c(w_\alpha\otimes w_\beta)=\omega^{8nks}w_\beta\otimes w_\alpha$ for 
$\alpha, \beta\in\overline{1,2}$.
\end{proof}

\begin{lemma}\label{Nichols_C}
Denote $d=2k(s+t+1)+N(i+j)(t+1)$, 
\[
\dim \mathfrak{B}\left(\mathscr{C}_{ijk,p}^{st}\right)
=\left\{\begin{array}{ll} 
4, &N\mid d, 2N\nmid d, \text{(Cartan type $A_1\times A_1$)},\\ 
27, &2N\mid 3d, 2N\nmid d,  \text{(Cartan type $A_2$)},\\
\infty, & otherwise. 
\end{array}\right.
\]
\end{lemma}
\begin{proof}
$c(w_\alpha\otimes w_\beta)=q w_\beta\otimes w_\alpha$ for $\alpha, \beta\in\overline{1,2}$, 
$q=(-1)^{(i+j)(t+1)}\omega^{8nk(s+t+1)}$.  \\
$\dim \mathfrak{B}\left(\mathscr{C}_{ijk,p}^{st}\right)<\infty\iff$  
$\mathfrak{B}\left(\mathscr{C}_{ijk,p}^{st}\right)$
is of Cartan type $A_1\times A_1$ or $A_2$.
\end{proof}

\begin{lemma}\label{Nichols_D}
Denote $\alpha=8nk(s+t+1)-2jN(t+1)$, $\beta=8nk(s+t+1)+2jN(t+1)$.
$
\dim \mathfrak{B}\left(\mathscr{D}_{jk,p}^{st}\right)<\infty
$ 
if and only if one of the following conditions holds. 
\begin{enumerate}
\item   $8nN\nmid \alpha$, $8nN\mid 2\beta$, Cartan type $A_1\times A_1$;
\item $8nN\nmid \alpha$, $8nN\mid (\alpha+2\beta)$, Cartan type $A_2$;
\item  $\alpha\equiv 4nN\mod 8nN$,  $2\beta\nequiv 0\,\text{and}\,4nN\mod 8nN$, 
Super type ${\bf A}_{2}(q;\I_2)$;
\item $\alpha-4\beta\equiv 12\beta\equiv 4nN\mod 8nN$, $8\beta\nequiv 0\mod 8nN$.
The Nichols algebras $\ufo(8)$, see \cite[Page 561]{Andruskiewitsch2017}.
\end{enumerate}
\end{lemma}

\begin{proof}
The braiding is  given by 
\begin{align*}
c(w_1\otimes w_1)
&=\omega^{\alpha}w_1\otimes w_1,&
c(w_1\otimes w_2)
&=\omega^{\beta}w_2\otimes w_1,\\
c(w_2\otimes w_1)
&=\omega^{\beta}w_1\otimes w_2,&
c(w_2\otimes w_2)
&=\omega^{\alpha}w_2\otimes w_2.
\end{align*}
\end{proof}

\begin{corollary}\label{proofUFO8}
If $\mathfrak{B}\left(\mathscr{D}_{jk,p}^{st}\right)$ is isomorphic to the Nichols algebras $\ufo(8)$, $5\nmid N$ and $17\nmid n$, then $8\mid N$, $4\mid n$.
\end{corollary}
\begin{remark}
When $n=8$, $N=48$, $j=2$ and $k\leq 11$, then $\mathfrak{B}\left(\mathscr{D}_{jk,p}^{st}\right)$ is isomorphic to the Nichols algebras $\ufo(8)$ in case that $(k,s,t)$ is in the set 
\begin{align*}
\left\{\begin{array}{l}
( 1 , 18 , 0 ),
( 1 , 22 , 2 ),
( 1 , 26 , 4 ),
( 1 , 30 , 6 ),
( 1 , 34 , 0 ),
( 1 , 38 , 2 ),
( 1 , 42 , 4 ),\\
( 1 , 46 , 6 ),
( 5 , 2 , 2 ),
( 5 , 6 , 0 ),
( 5 , 10 , 6 ),
( 5 , 14 , 4 ),
( 5 , 22 , 0 ),
( 5 , 30 , 4 ),\\
( 5 , 34 , 2 ),
( 5 , 42 , 6 ),
( 7 , 4 , 0 ),
( 7 , 12 , 6 ),
( 7 , 20 , 4 ),
( 7 , 28 , 2 ),
( 7 , 28 , 6 ),\\
( 7 , 36 , 0 ),
( 7 , 36 , 4 ),
( 7 , 44 , 2 ),
( 11 , 8 , 2 ),
( 11 , 8 , 4 ),
( 11 , 24 , 0 ),
( 11 , 24 , 4 ),\\
( 11 , 24 , 6 ),
( 11 , 40 , 0 ),
( 11 , 40 , 2 ),
( 11 , 40 , 6 )
\end{array}\right\}
\end{align*}
\end{remark}
\begin{proof}
Since $\alpha-4\beta\equiv 12\beta\equiv 4nN\mod 8nN$, 
\begin{align*}
&\quad \left\{\begin{array}{l}
15\times 8nk(s+t+1)+17\times 2jN(t+1)\equiv 0 \mod 8nN,\\
12\times 8nk(s+t+1)+12\times 2jN(t+1)\equiv 4nN \mod 8nN,
\end{array}\right.\\
&\Rightarrow 8n\mid 17\times 2jN(t+1),\quad 4N\mid 15\times 8nk(s+t+1),\\
&\Rightarrow   17\times 2jN(t+1)=8nr_1, \quad  15\times 8nk(s+t+1)=4Nr_2,\\
&\Rightarrow 
\left\{\begin{array}{l}
12\times 8nk(s+t+1)+12\times \frac{8nr_1}{17}\equiv 4nN \mod 8nN,\\
12\cdot \frac{4Nr_2}{15}+24jN(t+1)\equiv 4nN\mod 8nN,\\
\end{array}\right.\\
&\Rightarrow 32n\mid 4nN,\quad 16N\mid 4nN,\\
&\Rightarrow 8\mid N,\quad 4\mid n.
\end{align*}
\end{proof}

\begin{lemma}\label{Nichols_E}
Denote $\alpha=8nk(s+t)+2jNt$, $\beta=8nk(s+t)-2jNt$.
$
\dim \mathfrak{B}\left(\mathscr{E}_{jk,p}^{st}\right)<\infty
$ 
if and only if one of the following conditions holds. 
\begin{enumerate}
\item   $8nN\nmid \alpha$, $8nN\mid 2\beta$, Cartan type $A_1\times A_1$;
\item $8nN\nmid \alpha$, $8nN\mid (\alpha+2\beta)$, Cartan type $A_2$;
\item  $\alpha\equiv 4nN\mod 8nN$,  $2\beta\nequiv 0\,\text{and}\,4nN\mod 8nN$, 
Super type ${\bf A}_{2}(q;\I_2)$;
\item $\alpha-4\beta\equiv 12\beta\equiv 4nN\mod 8nN$, $8\beta\nequiv 0\mod 8nN$.
The Nichols algebras $\ufo(8)$.
\end{enumerate}
\end{lemma}

\begin{proof}
The braiding is  given by 
\begin{align*}
c(w_1\otimes w_1)
&=\omega^{\alpha}w_1\otimes w_1,&
c(w_1\otimes w_2)
&=\omega^{\beta}w_2\otimes w_1,\\
c(w_2\otimes w_1)
&=\omega^{\beta}w_1\otimes w_2,&
c(w_2\otimes w_2)
&=\omega^{\alpha}w_2\otimes w_2.
\end{align*}
\end{proof}

\begin{corollary}
If $\mathfrak{B}\left(\mathscr{E}_{jk,p}^{st}\right)$ is isomorphic to the Nichols algebras $\ufo(8)$, 
$5\nmid N$ and $17\nmid n$, then $8\mid N$, $4\mid n$.
\end{corollary}
\begin{remark}
The proof is similar to  the Corollary \ref{proofUFO8}. 
If $n=8$, $N=48$, $j=2$ and $k\leq 11$, then $\mathfrak{B}\left(\mathscr{E}_{jk,p}^{st}\right)$ is isomorphic to the Nichols algebras $\ufo(8)$ in case that $(k,s,t)$ is in the set 
\begin{align*}
\left\{\begin{array}{l}
( 1 , 4 , 3 ),
( 1 , 4 , 7 ),
( 1 , 12 , 1 ),
( 1 , 12 , 5 ),
( 1 , 20 , 3 ),
( 1 , 28 , 1 ),
( 1 , 36 , 7 ),\\
( 1 , 44 , 5 ),
( 5 , 8 , 3 ),
( 5 , 8 , 5 ),
( 5 , 24 , 1 ),
( 5 , 24 , 5 ),
( 5 , 24 , 7 ),
( 5 , 40 , 1 ),\\
( 5 , 40 , 3 ),
( 5 , 40 , 7 ),
( 7 , 2 , 5 ),
( 7 , 6 , 7 ),
( 7 , 10 , 1 ),
( 7 , 14 , 3 ),
( 7 , 18 , 5 ),\\
( 7 , 22 , 7 ),
( 7 , 42 , 1 ),
( 7 , 46 , 3 ),
( 11 , 2 , 3 ),
( 11 , 6 , 1 ),
( 11 , 10 , 7 ),
( 11 , 14 , 5 ),\\
( 11 , 22 , 1 ),
( 11 , 30 , 5 ),
( 11 , 34 , 3 ),
( 11 , 42 , 7 )
\end{array}\right\}
\end{align*}
\end{remark}

\subsection{The Nichols algebras of type $V_{abe}$}\label{typeVabe}
Let $V_{abe}=\k v_1\oplus \k v_2$ be a vector space with a braiding given by 
\begin{align*}
c(v_1\otimes v_1)&=a v_2\otimes v_2,\quad 
&c(v_1\otimes v_2)&=b   v_1\otimes v_2,\\
c(v_2\otimes v_1)&=b v_2\otimes v_1,\quad 
&c(v_2\otimes v_2)&=e   v_1\otimes v_1,
\end{align*}
then we call that the braided vector space $(V_{abe}, c)$ is of type $V_{abe}$. The braided vector space $V_{abe}$ is isomorphic to  $V_{ae\,b\,1}$ via $v_1\mapsto \sqrt{e}v_1$, $v_2\mapsto v_2$. If $ae=b^2$, then 
$V_{abe}$ is of diagonal type and
\[
\dim \mathfrak{B}(V_{abe})
=\left\{\begin{array}{ll}
4, & b=-1,\quad (\text{Cartan type $A_1\times A_1$}),\\
27, & b^3=1\neq b,\quad  (\text{Cartan type $A_2$}),\\
\infty, & \text{otherwise}.
\end{array}\right.
\]
If $b^2\neq ae$, $\mathfrak{B}(V_{abe})$ is obviously not of rack type,  please refer to 
the formula \eqref{fomulaeVabe} for more details.

\begin{lemma}\label{Nichols_G}
$\mathfrak{B}\left(\mathscr{G}_{jk,p}^{st}\right)$ is of type $V_{abe}$, where 
\begin{align*}
ae=\bar{\mu}^{2s+2t+1}\omega^{4kn(4s+4t+2)+jN(-2-4t)},\quad
b=(-1)^p\bar{\mu}^{s+t+\frac12}\omega^{4nk(2s+2t+1)+jN(2t+1)}.
\end{align*}
\end{lemma}
\begin{remark}
Notice that $\frac{ae}{b^2}=\omega^{-4jN(2t+1)}$, we have 
 $b=\omega^{2jN(2t+1)}$ for some suitable $p\in\Bbb{Z}_2$ under the case 
 $ae=1$. Furthermore, if we take $t=0$ and 
 $
j=\left\{\begin{array}{ll}
 2,& \lambda=1,\\
1,& \lambda=-1,\\
 \end{array}\right.
 $
 then 
  $
b\in \left\{\begin{array}{ll}
\Bbb{G}_{2n},& \lambda=1,\\
\Bbb{G}_{4n},& \lambda=-1.\\
 \end{array}\right.
 $
 So 
 $\dim \mathfrak{B}\left(\mathscr{G}_{jk,p}^{st}\right)=
 \left\{\begin{array}{ll}
 (2n)^2,& \lambda=1,\\
 (4n)^2,& \lambda=-1,\\
 \end{array}\right.
 $
 for suitable choice of  $(n,N,j,s,t,k,p)$.
 
Similarly, if $b=-1$, then 
 $\dim \mathfrak{B}\left(\mathscr{G}_{jk,p}^{st}\right)=
 \left\{\begin{array}{ll}
4n,& \lambda=1,\\
8n,& \lambda=-1,\\
 \end{array}\right.
 $
 for suitable choice of  $(n,N,j,s,t,k,p)$.
\end{remark}

\begin{corollary}\label{Vabe}
Suppose $ae\neq b^2=(ae)^{-1}$ and $b\in\Bbb{G}_{n}$ for $n\geq 5$, then 
\[
\dim \mathfrak{B}(V_{abe})=\infty.
\]
\end{corollary}
\begin{proof}
$\mathfrak{B}\left(\mathscr{G}_{jk,p}^{st}\right)$ is of type $V_{abe}$, 
and $b=(-1)^p\omega^{j(2t+1)}$, $ae=\omega^{-2j(2t+1)}$ in case 
$\mu=\lambda=1=N$. According to  \cite{andruskiewitsch2007pointed}, 
Nichols algebras  associated with two dimensional Yetter-Drinfeld modules, over  the dihedral group $D_{4n}$ of order $4n$,  are either $4$-dimension or infinite dimension. 
And $A_{1n}^{++}$ is isomorphic to a $2$-cocycle deformation of $\k^{D_{4n}}$
\cite{MR1800713}, so $\dim \mathfrak{B}\left(\mathscr{G}_{jk,p}^{st}\right)=\infty$ under the provided conditions. 
If we  take $j=8$, then $b=(-1)^p \omega^{8(2t+1)}$ covers all $n$-th primitive roots of unity for $n\geq 5$.  
\end{proof}
\begin{remark}
The corollary provided a correct proof for 
\[
\dim\mathfrak{B}\left(W_1^a\right)=\infty=\dim\mathfrak{B}\left(W_2^a\right),\] 
see \cite[Page 278]{Shi2019}, where the braiding of $\mathfrak{B}\left(W_1^a\right)$
should be corrected as 
\begin{align*}
&c\left(w_1^{(1)}\otimes w_1^{(1)}\right)=
-\theta w_2^{(1)}\otimes w_2^{(1)},\quad
c\left(w_1^{(1)}\otimes w_2^{(1)}\right)=
-\theta w_1^{(1)}\otimes w_2^{(1)},\\
&c\left(w_2^{(1)}\otimes w_1^{(1)}\right)=
-\theta w_2^{(1)}\otimes w_1^{(1)},\quad
c\left(w_2^{(1)}\otimes w_2^{(1)}\right)=
\theta w_1^{(1)}\otimes w_1^{(1)}.
\end{align*}
The parameters $\theta=\pm \frac{\sqrt{2}\left(\sqrt{-1}-1\right)}{2}$ are $8$-th primitive roots of unity. 
\end{remark}

\begin{lemma}\label{Nichols_H}
$\mathfrak{B}\left(\mathscr{H}_{jk,p}^{s}\right)$  is  of type $V_{abe}$ with 
\[
ae=\bar{\mu}^{2s+2t+1}\omega^{4kn(4t+4s+2)+jN(4t+2)},\quad
b=(-1)^p\bar{\mu}^{s+t+\frac12}\omega^{4kn(2t+2s+1)+jN(-1-2t)}.
\]
\end{lemma}
\begin{remark}
From observation, we have $\frac{ae}{b^2}=\omega^{4jN(2t+1)}$. 
\begin{enumerate}
\item Suppose  $ae=1$, then 
$
\dim \mathfrak{B}\left(\mathscr{H}_{jk,p}^{s}\right)
=\left\{\begin{array}{ll}
(4n)^2, & \lambda=-1,\\
(2n)^2, & \lambda=1,
\end{array}\right.
$
under suitable choice of  $(n,N,j,s,t,k,p)$.
\item Suppose  $b=-1$, then 
$
\dim \mathfrak{B}\left(\mathscr{H}_{jk,p}^{s}\right)
=\left\{\begin{array}{ll}
8n, & \lambda=-1,\\
4n, & \lambda=1,
\end{array}\right.
$
under suitable choice of  $(n,N,j,s,t,k,p)$.
\end{enumerate}
\end{remark}

\begin{lemma}\label{Nichols_P}
Denote $q=(-1)^{p+(i+j)(t+\frac12)+j}\tilde{\mu}^{2s+2t+1}\omega^{4nk(2s+2t+1)}$, then 
\[
\dim \mathfrak{B}\left(\mathscr{P}_{ijk,p}^{st}\right)
=\left\{\begin{array}{ll}
4, & q=-1, (\text{Cartan type $A_1\times A_1$}),\\
27, & q^3=1\neq q,  (\text{Cartan type $A_2$}),\\
\infty, & \text{otherwise}.
\end{array}\right.
\]
\end{lemma}

\begin{proof}
$\mathscr{P}_{ijk,p}^{st}$  is of type $V_{qqq}$, 
so $\mathfrak{B}\left(\mathscr{P}_{ijk,p}^{st}\right)$ is finite 
dimensional iff  $q=-1$ or $q^3=1\neq q$. 
\end{proof}

\subsection{The Nichols algebras over $\mathscr{I}_{pjk}^{s}$}\label{section_Nichos_I}
\begin{lemma}\label{racklemma}
Let $(V,c)$ be a braided vector space such that 
$c(x\otimes y)\in \k f_x(y)\otimes x$, where the map 
$f_x: V\to V$ is bijective for any $x\in V$ under a fixed basis. Then $(V,\vartriangleright)$ is a rack 
with $x\vartriangleright y=f_x(y)$ under the fixed basis. 
\end{lemma}
\begin{proof}
For any $x,y,z$ in a fixed basis of $V$,
\begin{align*}
c_1c_2c_1(x\otimes y\otimes z)
&\in \k [(x\vartriangleright y)\vartriangleright (x \vartriangleright z)]
\otimes (x\vartriangleright y)\otimes x, \\
c_2c_1c_2(x\otimes y\otimes z)
&\in\k [x\vartriangleright (y\vartriangleright z)]
\otimes (x\vartriangleright y)\otimes x. 
\end{align*}
So $x\vartriangleright (y\vartriangleright z)
=(x\vartriangleright y)\vartriangleright (x \vartriangleright z)$.
\end{proof}

\begin{lemma} \label{Nichols_I_rack}
If $n>2$, then $\dim  \mathfrak{B}\left(\mathscr{I}_{pjk}^{s}\right)=\infty$.
\end{lemma}
\begin{proof}
In case $n>2$,
let $X=\{w_r, m_r\mid r\in\overline{1,n}\}$, then $(X,\vartriangleright)$
is a rack as defined in Lemma \ref{racklemma}. It's easy to see that 
$\{w_r\mid r\in\overline{1,n}\}$ and $\{m_r\mid r\in\overline{1,n}\}$
are two subracks of $X$.  $X$ is of type $D$, since 
\[
w_1\vartriangleright (m_1\vartriangleright (w_1 \vartriangleright m_1))
=\left\{\begin{array}{ll}
m_4,&\text{if}\,\, n>3,\\
m_3,&\text{if}\,\, n=3.
\end{array}\right.
\]
According to \cite[Theorem 3.6]{Andruskiewitsch2011}, $\dim  \mathfrak{B}\left(\mathscr{I}_{pjk}^{s}\right)=\infty$.
\end{proof}

\begin{lemma} \label{Nichols_I}
Let  $q=(-1)^p\omega^{4kn(2s+1)}$.
\begin{enumerate}
\item If $n=1$, then 
\[
\dim  \mathfrak{B}\left(\mathscr{I}_{pjk}^{s}\right)<\infty
\iff \left\{\begin{array}{l}
\lambda q^2=1\neq q, (\text{Cartan type}\,\, A_1\times A_1),\\
\lambda q^3=1\neq q, (\text{Cartan type}\,\, A_2).
\end{array}\right.
\]
\item If $n=2$, then 
\[\dim \mathfrak{B}\left(\mathscr{I}_{pjk}^{s}\right)
=\left\{\begin{array}{ll} 64, & q=-1, j=2\,\text{and}\, \lambda=1,  
(\text{dihedral rack type $\Bbb{D}_4$}), \\
64, & q=-1, j=4\,\text{and}\, \lambda=1,  
(\text{Cartan type $A_2\times A_2$}), \\
\infty, &\text{otherwise}. 
\end{array}\right.
\]
\end{enumerate}
\end{lemma}
\begin{proof}
If $n=1$, then the braiding of $\mathscr{I}_{pjk}^{s}$ is given by 
\begin{align*}
c(w_1\otimes w_1)
&
=q w_1\otimes w_1, &
c(w_1\otimes m_1)
&
=\lambda q\omega^{2njN(2s+1)} m_1\otimes w_1,\\
c(m_1\otimes w_1)
&
=q\omega^{2njN}w_1\otimes m_1 ,&
c(m_1\otimes m_1)
&
=q m_1\otimes m_1.
\end{align*}
So $\mathfrak{B}\left(\mathscr{I}_{pjk}^{s}\right)$ is of diagonal type and its Dynkin diagram is 
   $\xymatrix{ \overset{q}{\underset{\ }{\circ}}\ar  @{-}[rr]^{\lambda q^2}  &&
\overset{q}{\underset{\ }{\circ}}}$ if  $\lambda q^2\neq 1$. 

If $n=2$, then the braiding of $\mathscr{I}_{pjk}^{s}$ is given by 
\begin{align*}
c(w_1\otimes w_1)&=q w_1\otimes w_1,  &
c(w_1\otimes w_2)&=q\beta w_2\otimes w_1,\\
c(w_2\otimes w_1)&=q\beta w_1\otimes w_2,  &
c(w_2\otimes w_2)&=q w_2\otimes w_2,\\
c(w_1\otimes m_1)&=\alpha m_2\otimes w_1,  &
c(w_1\otimes m_2)&=q^2\alpha^{-1} m_1\otimes w_1,\\
c(w_2\otimes m_1)&=\lambda \alpha\beta m_2\otimes w_2,  &
c(w_2\otimes m_2)&=\lambda q^2 \alpha^{-1}\beta m_1\otimes w_2,\\
c(m_1\otimes m_1)&=q m_1\otimes m_1,  &
c(m_1\otimes m_2)&=q\lambda \beta m_2\otimes m_1,\\
c(m_2\otimes m_1)&=q\lambda \beta m_1\otimes m_2,  &
c(m_2\otimes m_2)&=q m_2\otimes m_2,\\
c(m_1\otimes w_1)&=\alpha  w_2\otimes m_1,&
c(m_1\otimes w_2)&=q^2\alpha^{-1}w_1\otimes m_1,\\
c(m_2\otimes w_1)&=\alpha\beta w_2\otimes m_2,&
c(m_2\otimes w_2)&=q^2\alpha^{-1}\beta w_1\otimes m_2, 
\end{align*}
where $\alpha=\omega^{8kns}$, $\beta=\omega^{2njN}=\pm 1$. 
Both $\k w_1\oplus \k w_2$ and $\k m_1\oplus \k m_2$ are braided subspaces of diagonal type. 
If $q^2\neq 1$, then their Dynkin diagrams are given by 
$\xymatrix{ \overset{q}{\underset{\ }{\circ}}\ar  @{-}[rr]^{ q^2}  &&
\overset{q}{\underset{\ }{\circ}}}$. So $\dim \mathfrak{B}\left(\mathscr{I}_{pjk}^{s}\right)<\infty$
iff          $q=-1$ or $q^3=1\neq q$.

It's easy to see $\mathfrak{B}\left(\mathscr{I}_{pjk}^{s}\right)
\simeq \mathfrak{B}(\Bbb{D}_4,c_q)$, where  $\Bbb{D}_4$ is the dihedral rack and 
$c_q$ is some $2$-cocycle over $\Bbb{D}_4$. According to \cite{Heckenberger2017}, 
if  $\dim \mathfrak{B}\left(\mathscr{I}_{pjk}^{s}\right)<\infty$, then 
$\dim \mathfrak{B}\left(\mathscr{I}_{pjk}^{s}\right)=64$. This only could be happened in case 
$\lambda=1$ and $q=-1$. The relations of the $64$-dimensional Nichols algebras are given by 
\begin{align*}
w_1w_2+\beta w_2w_1=0,\quad   w_1^2=w_2^2=0,\quad
m_1m_2+\beta m_2m_1=0,\quad   m_1^2=m_2^2=0,\\
w_1m_1-\alpha  m_2w_1+\alpha ^2\beta  w_2m_2-\alpha m_1w_2=0,\\
w_1m_2-\alpha ^{-1} m_1w_1+w_2m_1-\alpha \beta m_2w_2=0,\\
w_1m_1w_1m_1+\beta m_1w_1m_1w_1=0, \\
w_1m_1w_2m_1+w_2m_1w_1m_1+m_1w_1m_1w_2+m_1w_2m_1w_1=0.
\end{align*}
In particular, if $j=4$, then the Nichols algebra is of Cartan type $A_2\times A_2$, which 
was  appeared  first in \cite[Example 6.5]{Milinski2000}. Denote 
$u_1=w_1+\alpha w_2$, $u_2=m_1-\alpha m_2$, $u_3=w_1-\alpha w_2$, $u_4=m_1+\alpha m_2$,
we can see this from the  relations: $u_i^2=0$ for $i\in\{1,2,3,4\}$ and 
\begin{align*}
(u_1u_2)^2+(u_2u_1)^2=0,
\quad u_1u_3+u_3u_1=0, \quad u_1u_4-u_4u_1=0,\\
(u_3u_4)^2+(u_4u_3)^2=0,\quad  
u_3u_2+u_2u_3=0, \quad u_2u_4+u_4u_2=0.
\end{align*}
\end{proof}

\subsection{The Nichols algebras over $\mathscr{K}_{jk,p}^{s}$}\label{section_Nichos_K}
Let 
\begin{align*}
b+2a-2=2nr+d,\quad r\in\Bbb{N},\quad 0\leq d\leq 2n-1,\\
2n+1-b+2a-2=2ne+f, \quad e\in\Bbb{N},\quad 0\leq f\leq 2n-1,
\end{align*}
then the braiding of $\mathscr{K}_{jk,p}^{s}$ is given by 
\begin{align*}
&\quad c(w_a\otimes w_b)=\\
&\left\{\begin{array}{ll}
(-1)^p\left(\bar{\mu}\omega^{8nk}\right)^s w_b\otimes w_1, & a=1,\\
(-1)^p\lambda^r\bar{\mu}^{s+n(r-2)}
\omega^{2n(r-2)(4nk+jN)+8nks} w_{2n}\otimes w_{2n-a+2}, &a>1, d=0, 2\mid (a+b),\\
(-1)^p\lambda^{r+1}\bar{\mu}^{s+n(r-1)}
\omega^{2n(r-1)(4nk+jN)+8nks} w_{d}\otimes w_{2n-a+2}, &a>1, d>0, 2\mid (a+b),\\
(-1)^p\lambda^e\left(\bar{\mu}\omega^{8nk}\right)^{s-ne-2+2a}
\omega^{-2jNne}w_{1}\otimes w_{2n-a+2}, &a>1, f=0, 2\nmid (a+b),\\
\frac{(-1)^p\lambda^{e+1}\left(\bar{\mu}\omega^{8nk}\right)^{s-n(e+1)-2+2a}}
{\omega^{2jNn(e+1)}}w_{2n+1-f}\otimes w_{2n-a+2}, &a>1, f>0, 2\nmid (a+b).
\end{array}\right. 
\end{align*}

\subsubsection{The Nichols algebras over $\mathscr{K}_{jk,p}^{s}$ for $n=1$, $2$}
\begin{lemma}\label{Nichols_K}
Let $q=(-1)^p\bar{\mu}^s\omega^{8kns}$.
\begin{enumerate}
\item If $n=1$, then 
\[
\dim \mathfrak{B}\left(\mathscr{K}_{jk,p}^{s}\right)<\infty\iff
\left\{\begin{array}{l}
\lambda q^2=1\neq q,\quad (\text{Cartan type}\,\, A_1\times A_1),\\
\lambda q^3=1\neq q, \quad (\text{Cartan type}\,\, A_2).
\end{array}\right.
\]
\item If $n=2$, then $\dim \mathfrak{B}\left(\mathscr{K}_{jk,p}^{s}\right)<\infty\Longrightarrow q=-1$ or $q^3=1\neq q$. In particular, if
$\lambda=1$, then 
$\dim \mathfrak{B}\left(\mathscr{K}_{jk,p}^{s}\right)=
\left\{\begin{array}{ll}
64, &q=-1,\\
\infty, &\text{otherwise}.
\end{array}\right.$
\end{enumerate}
\end{lemma}
\begin{proof}
If $n=1$, then the braiding of $\mathfrak{B}\left(\mathscr{K}_{jk,p}^{s}\right)$ is given by 
\begin{align*}
c( w_1 \otimes w_1 )&
                                    =q w_1 \otimes w_1,&
c( w_1 \otimes w_2 )&
                                    =q w_2 \otimes w_1,\\
c( w_2 \otimes w_1 )&
                                   =q\omega^{-4jN}w_1 \otimes w_2,&
c( w_2 \otimes w_2 )&
                                    =q w_2 \otimes w_2.
\end{align*}
$\mathfrak{B}\left(\mathscr{K}_{jk,p}^{s}\right)$ is of diagonal type and 
its Dynkin diagram is    $\xymatrix{ \overset{q}{\underset{\ }{\circ}}\ar  @{-}[rr]^{ \lambda q^2}  &&
\overset{q}{\underset{\ }{\circ}}}$ in case  $\lambda q^2\neq 1$. 

Denote $a=\bar{\mu}^2\omega^{16kn}$, 
$b=\omega^{4jN}$. 
If $n=2$,  then the braiding of $\mathfrak{B}\left(\mathscr{K}_{jk,p}^{s}\right)$ is given by 
\begin{align*}
c( w_1 \otimes w_1 )&= q w_1 \otimes w_1, &
c( w_1 \otimes w_2 )&= q w_2 \otimes w_1, \\
c( w_1 \otimes w_3 )&= q w_3 \otimes w_1, &
c( w_1 \otimes w_4 )&= q w_4 \otimes w_1, \\
c( w_2 \otimes w_1 )&= qa^{-1}b^{2} w_3 \otimes w_4, &
c( w_2 \otimes w_2 )&= q\lambda a^{-1}b^{-1} w_4 \otimes w_4, \\
c( w_2 \otimes w_3 )&=q\lambda b^{-1} w_1 \otimes w_4, &
c( w_2 \otimes w_4 )&=  q w_2 \otimes w_4, \\
c( w_3 \otimes w_1 )&=  q w_1 \otimes w_3, &
c( w_3 \otimes w_2 )&= qb^2 w_2 \otimes w_3, \\
c( w_3 \otimes w_3 )&=  q w_3 \otimes w_3, &
c( w_3 \otimes w_4 )&= qb^2 w_4 \otimes w_3, \\
c( w_4 \otimes w_1 )&=  q\lambda b w_3 \otimes w_2, &
c( w_4 \otimes w_2 )&=  q w_4 \otimes w_2, \\
c( w_4 \otimes w_3 )&=  qab^2 w_1 \otimes w_2, &
c( w_4 \otimes w_4 )&=  q\lambda ab w_2 \otimes w_2.
\end{align*}
$W_1=\k w_1\oplus \k w_3$ and $W_2=\k w_2\oplus \k w_4$  are braided vector spaces. $\mathfrak{B}(W_1)$
is of diagonal type and its Dynkin diagram is 
   $\xymatrix{ \overset{q}{\underset{\ }{\circ}}\ar  @{-}[rr]^{ q^2}  &&
\overset{q}{\underset{\ }{\circ}}}$ if $q^2\neq 1$. 
$\mathfrak{B}(W_1)$ (or $\mathfrak{B}(W_2)$) 
  is finite dimensional iff $q=-1$ or $q^3=1\neq q$.

In case $\lambda=1$, then $b^2=1$. 
Denote $u_i=w_2+(-1)^{i+1}\sqrt{\frac{1}{ab}}w_4$ for $i=1, 2$, and 
$u_j=w_1+(-1)^{j+1}\sqrt{\frac{1}{ab}}w_3$ for $j=3, 4$. 
Then the Nichols algebra is diagonal type and its  Dynkin diagram is given by 
the Figure \ref{figureEightM23} if $q^2\neq \pm 1$. 
According to Heckenberger's classification result \cite{heckenberger2009classification}, 
we have  $\dim \mathfrak{B}\left(\mathscr{K}_{jk,p}^{s}\right)=\infty$ in case $\lambda=1$ and 
$q\neq -1$. 
\begin{figure}[h!]
$$
\xy 
(0,0)*\cir<4pt>{}="E1", 
(30,0)*\cir<4pt>{}="E2",
(30,30)*\cir<4pt>{}="E3",
(0,30)*\cir<4pt>{}="E4",
(-4,0)*+{q},
(-4,15)*+{q^2b},
(-4,30)*+{q},
(34,15)*+{q^2b},
(34,0)*+{q},
(34,30)*+{q},
(14,33)*+{q^2},
(15,-3)*+{q^2},
(5,18)*+{-bq^2},
(24,18)*+{-bq^2},
\ar @{-}"E1";"E2"
\ar @{-}"E2";"E3"
\ar @{-}"E4";"E1"
\ar @{-}"E4";"E3"
\ar @{-}"E1";"E3"
\ar @{-}"E2";"E4"
\endxy
$$
\caption{}
\label{figureEightM23}
\end{figure}
In case $q=-1$ and $\lambda=1$, it's easy to see  that 
$\mathfrak{B}\left(\mathscr{K}_{jk,p}^{s}\right)$
is of Cartan type $A_2\times A_2$ and $\dim \mathfrak{B}\left(\mathscr{K}_{jk,p}^{s}\right)=64$.
\end{proof}

\begin{remark}
The relations of the $64$-dimensional Nichols 
algebra of Cartan type $A_2\times A_2$ is given by 
\begin{align*}
w_1^2=0,\quad w_3^2=0,\quad w_1w_3+w_3w_1=0,\\
w_2w_4=0,\quad w_4w_2=0,\quad w_2^2+\frac{ b}{a}w_4^2=0,\\
w_1w_2+w_2w_1+a^{-1} w_3w_4+a^{-1}w_4w_3=0, \\
w_1w_4+w_4w_1+ b w_3w_2+ b w_2w_3=0, \\
w_1w_2w_3w_4=a w_2w_1w_2w_1,\\
w_2w_3w_2w_1=b w_1w_2w_1w_4+\frac1a w_3w_2w_3w_4-w_2w_1w_2w_3,
\end{align*}
where $a=\bar{\mu}^2\omega^{16kn}$, 
$b=\omega^{4jN}=\pm 1$.  
\end{remark}

\subsubsection{The Nichols algebras over $\mathscr{K}_{jk,p}^{s}$ for $n=3$}
\label{NicholsKCasen3}
Denote $q=(-1)^p\bar{\mu}^s\omega^{8kns}$, 
$\alpha=\lambda\bar{\mu}\omega^{8kn}$, $\beta=\omega^{6jN}$, then the braiding of 
$\mathscr{K}_{jk,p}^{s}$ is given by 
\begin{align*}
c( w_1 \otimes w_1 )&=q w_1 \otimes w_1, &
c( w_1 \otimes w_2 )&=q w_2 \otimes w_1, \\
c( w_1 \otimes w_3 )&=q w_3 \otimes w_1, &
c( w_1 \otimes w_4 )&=q w_4 \otimes w_1,\\
c( w_1 \otimes w_5 )& =q w_5 \otimes w_1,&
c( w_1 \otimes w_6 )& =q w_6 \otimes w_1,\\
c( w_2 \otimes w_1 )& =\frac{q}{\beta^2  \alpha^4}w_5 \otimes w_6,&
c( w_2 \otimes w_2 )& =\frac{q}{\beta  \alpha^3}w_4 \otimes w_6,\\
c( w_2 \otimes w_3 )& =\frac{q}{\beta \alpha}w_1 \otimes w_6,&
c( w_2 \otimes w_4 )&=\frac{q}{\beta \alpha^3} w_6 \otimes w_6,\\
c( w_2 \otimes w_5 )&=\frac{q}{\beta \alpha} w_3 \otimes w_6,&
c( w_2 \otimes w_6 )&=q w_2 \otimes w_6, \\
c( w_3 \otimes w_1 )&=\frac{q}{\beta \alpha^3} w_5 \otimes w_5,&
c( w_3 \otimes w_2 )&=\frac{q}{\beta^2  \alpha^2} w_4 \otimes w_5,\\
c( w_3 \otimes w_3 )&=q w_1 \otimes w_5, &
c( w_3 \otimes w_4 )&=\frac{q}{\beta^2  \alpha^2} w_6 \otimes w_5, \\
c( w_3 \otimes w_5 )&=q w_3 \otimes w_5, &
c( w_3 \otimes w_6 )&=\frac{q\alpha}{\beta } w_2 \otimes w_5, \\
c( w_4 \otimes w_1 )&=\frac{q}{\beta^2  } w_1 \otimes w_4, &
c( w_4 \otimes w_2 )&=q w_2 \otimes w_4, \\
c( w_4 \otimes w_3 )&=\frac{q}{\beta^2  }w_3 \otimes w_4, &
c( w_4 \otimes w_4 )&=q w_4 \otimes w_4,  \\
c( w_4 \otimes w_5 )&=\frac{q}{\beta^2  }w_5 \otimes w_4, &
c( w_4 \otimes w_6 )&=q w_6 \otimes w_4, \\
c( w_5 \otimes w_1 )&=q w_3 \otimes w_3,  &
c( w_5 \otimes w_2 )&=\frac{q}{\beta^3  \alpha} w_6 \otimes w_3, \\
c( w_5 \otimes w_3 )&=q w_5 \otimes w_3,  &
c( w_5 \otimes w_4 )&=\frac{q\alpha^2}{\beta^2  } w_2 \otimes w_3, \\
c( w_5 \otimes w_5 )&=\beta q\alpha^3 w_1 \otimes w_3, &
c( w_5 \otimes w_6 )&=\frac{q\alpha^2}{\beta^2  }w_4 \otimes w_3, \\
c( w_6 \otimes w_1 )&=\frac{q\alpha}{\beta^3  }w_3 \otimes w_2,  &
c( w_6 \otimes w_2 )&=q  w_6 \otimes w_2, \\
c( w_6 \otimes w_3 )&=\frac{q\alpha}{\beta^3  } w_5 \otimes w_2, &
c( w_6 \otimes w_4 )&=\beta  q\alpha^3 w_2 \otimes w_2, \\
c( w_6 \otimes w_5 )&=\frac{q\alpha^4}{\beta^2  } w_1 \otimes w_2, &
c( w_6 \otimes w_6 )&=\beta  q\alpha^3 w_4 \otimes w_2.
\end{align*}
If $q=-1$,  $\beta^2=1$, then the Nichols algebra has the folllowing relations:
\begin{align*}
w_1w_2+w_2w_1+\frac1{\alpha^4}w_5w_6
+\frac1{\alpha^4}w_6w_5+\frac1{\alpha^2}w_3w_4+\frac1{\alpha^2}w_4w_3=0,\\
w_1w_6+w_6w_1+\alpha\beta  w_2w_3+\alpha\beta w_3w_2
+\frac {\beta}{\alpha}  w_4w_5+\frac {\beta}{\alpha}  w_5w_4=0,\\
w_6w_2=w_5w_3=w_4^2=w_3w_5=w_2w_6=w_1^2=0,\\
 w_1w_4+w_4w_1=0,\quad
 w_2w_5+\frac{1}{\alpha\beta}  w_3w_6=0,\quad
 w_5w_2+\frac{1}{\alpha\beta}  w_6w_3=0,\\
w_1w_5+w_5w_1+w_3^2=0,\quad
w_1w_3+w_3w_1+\frac{1}{\alpha^3\beta}  w_5w_5=0,\\
w_2^2+\frac{1}{\alpha^3\beta}  w_4w_6+\frac{1}{\alpha^3\beta}  w_6w_4=0,\quad
w_2w_4+w_4w_2+\frac{1}{\alpha^3\beta}  w_6^2=0.
\end{align*}
\section{Appendix}
Here is a list of Yetter-Drinfeld modules over $A_{N\,2n}^{\mu\lambda}$ with structures decided by 
the formulae \ref{eq:action} and \ref{eq:coaction}. 
\subsection{One and two dimensional Yetter-Drinfeld modules over $A_{N\,2n}^{\mu\lambda}$}
\begin{enumerate}
\item $\mathscr{A}_{iik,p}^s=\k w$, where $w=v\boxtimes\left[x_{11}^{2s}+(-1)^px_{12}^{2s}\right]$, 
         $s\in\overline{1,N}$, $p\in\Bbb{Z}_2$, $V_{iik}=\k v$.
\item $\bar{\mathscr{A}}_{ijk,p}^s=\k w$, where
      $
      w=v\boxtimes\left[x_{11}^{2s+1}\chi_{22}^{2n-1}
      +(-1)^p\sqrt{\lambda} x_{12}^{2s+1}\chi_{21}^{2n-1}\right],
       $ 
      $s\in\overline{1,N}$ , $p\in\Bbb{Z}_2$, $V_{ijk}=\k v$, 
      $\left\{\begin{array}{ll}
       i=j,&\lambda=1,\\
       i= j+1, &\lambda=-1,
       \end{array}\right.$ and $i,j\in\Bbb{Z}_2$.
\item $\mathscr{B}_{ijk}^s=\k w_1\oplus \k w_2$, where 
       $s\in\overline{1,N}$, $i,j\in\Bbb{Z}_2$, $i= j+1$, $V_{ijk}=\k v$, 
       \[
       w_1=v\boxtimes \left[x_{11}^{2s}+x_{12}^{2s}\right], \quad
       w_2=v\boxtimes \left[x_{11}^{2s}-x_{12}^{2s}\right].
       \] 

\item $\mathscr{C}_{ijk,p}^{st}=\k w_1\oplus \k w_2$, where 
         $s\in\overline{1,N}$, $t\in\overline{0,n-1}$,  $V_{ijk}=\k v$, $i,j,p\in\Bbb{Z}_2$,  
         \begin{align*}
        w_1&=v\boxtimes\left[x_{11}^{2s+1}\chi_{22}^{2t+1}
        +(-1)^p\sqrt{(-1)^{i+j}}x_{12}^{2s+1}\chi_{21}^{2t+1}\right],\\
        w_2&=v\boxtimes\left[x_{11}^{2s}\chi_{22}^{2t+2}
        +\frac{(-1)^p}{\sqrt{(-1)^{i+j}}}x_{12}^{2s}\chi_{21}^{2t+2}\right]. 
         \end{align*} 
  
\item $\mathscr{D}_{jk,p}^{st}=\k w_1\oplus \k w_2$, where 
        $V_{jk}=\k v_1\oplus\k v_2$, $s\in\overline{1,N} $, 
        $t\in\overline{0,n-1}$, $p\in\Bbb{Z}_2$,   
         \begin{align*}
         w_1&=v_1\boxtimes x_{11}^{2s+1}\chi_{22}^{2t+1}
                     +(-1)^p\omega^{jN-4kn}v_2\boxtimes x_{12}^{2s+1}\chi_{21}^{2t+1},\\
        w_2&=v_2\boxtimes x_{11}^{2s}\chi_{22}^{2t+2}
                    +(-1)^p\omega^{4kn-jN}v_1\boxtimes x_{12}^{2s}\chi_{21}^{2t+2}. 
        \end{align*}
     
\item  $\mathscr{E}_{jk,p}^{st}=\k w_1\oplus \k w_2$, where 
         $V_{jk}=\k v_1\oplus\k v_2$, $s\in\overline{1,N}$, 
         $t\in\overline{0,n-1}$, $p\in\Bbb{Z}_2$, 
         \begin{align*}
         w_1&=v_1\boxtimes x_{11}^{2s}\chi_{22}^{2t}
                    +(-1)^p\omega^{jN-4kn}v_2\boxtimes x_{12}^{2s}\chi_{21}^{2t},\\
         w_2&=v_2\boxtimes x_{11}^{2s}\chi_{11}^{2t}
                      +(-1)^p\omega^{4kn-jN}v_1\boxtimes x_{12}^{2s}\chi_{12}^{2t}.
         \end{align*} 
 
\item $\mathscr{F}_{jk,p}^s=\k w_1\oplus\k w_2$, where 
       $V_{jk}=\k v_1\oplus\k v_2$, $s\in\overline{1,N}$,  $p\in\Bbb{Z}_2$, 
        \begin{align*}
        w_1&=v_1\boxtimes x_{11}^{2s}+(-1)^p\omega^{jN-4kn}v_2\boxtimes x_{12}^{2s},\\ 
        w_2&=v_2\boxtimes x_{11}^{2s}+(-1)^p\omega^{4kn-jN} v_1\boxtimes x_{12}^{2s}.
       \end{align*}

\item  $\mathscr{G}_{jk,p}^{st}=\k w_1\oplus \k w_2$, where 
       $V_{jk}^\prime=\k v_1^\prime\oplus \k v_2^\prime$, 
       $s\in\overline{1,N}$, $t\in\overline{0,n-1}$, $p\in\Bbb{Z}_2$, 
        \begin{align*}
        w_1&=v_1^\prime\boxtimes x_{11}^{2s+1}\chi_{22}^{2t}
                   +\frac{(-1)^p\omega^{jN-4kn}}{\sqrt{\bar{\mu}}}
                   v_2^\prime\boxtimes x_{12}^{2s+1}\chi_{21}^{2t},\\
        w_2&=v_2^\prime\boxtimes x_{11}^{2s}\chi_{22}^{2t+1}
                   +(-1)^p\sqrt{\bar{\mu}}\omega^{4kn-jN}
                    v_1^\prime\boxtimes x_{12}^{2s}\chi_{21}^{2t+1}.
       \end{align*} 
  
\item $\mathscr{H}_{jk,p}^{st}=\k w_1\oplus \k w_2$ , where 
        $V_{jk}^\prime=\k v_1^\prime\oplus \k v_2^\prime$, 
        $s\in\overline{1,N}$, $t\in\overline{0,n-1}$, $p\in\Bbb{Z}_2$, 
        \begin{align*}
        w_1&=v_1^\prime\boxtimes x_{11}^{2s}\chi_{22}^{2t+1}
                    +\frac{(-1)^p\omega^{jN-4kn}}{\sqrt{\bar{\mu}}}
                     v_2^\prime\boxtimes x_{12}^{2s}\chi_{21}^{2t+1},\\
        w_2&=v_2^\prime\boxtimes x_{11}^{2s+1}\chi_{22}^{2t}
                    +(-1)^p\sqrt{\bar{\mu}}\omega^{4kn-jN}
                     v_1^\prime\boxtimes x_{12}^{2s+1}\chi_{21}^{2t}.
        \end{align*}
    
\item $\mathscr{P}_{ijk,p}^{st}=\k w_1\oplus \k w_2$, where $\lambda=1$, 
         $V_{ijk}^\prime=\k v$, 
         $s\in\overline{1,N}$, $t\in\overline{0,n-1}$, $p\in\Bbb{Z}_2$, 
         \begin{align*}
         w_1&=v\boxtimes \left[x_{11}^{2s+1}\chi_{22}^{2t}
                     +\frac{(-1)^p}{\sqrt{(-1)^{i+j}}}x_{12}^{2s+1}\chi_{21}^{2t}\right],\\
         w_2&=v\boxtimes \left[
                     x_{11}^{2s}\chi_{22}^{2t+1}+\sqrt{(-1)^{i+j}}(-1)^p x_{12}^{2s}\chi_{21}^{2t+1}\right]. 
         \end{align*}  
                              
\end{enumerate}

\subsection{$2n$-dimensional Yetter-Drinfeld modules over $A_{N\,2n}^{\mu\lambda}$}
\begin{enumerate}
\item Let  $V_{jk}=\k v_1\oplus\k v_2$, $s\in\overline{1,N}$, $p\in \Bbb{Z}_2$.  Denote
\begin{align*}
w_r&=\left\{\begin{array}{rl}
           \left[v_1+(-1)^p\omega^{2(jN-2kn)}v_2\right]\boxtimes x_{11}^{2s+1}, & r=1, \\
           \chi_{22}^{r-1}\cdot w_1,&r\,\,\text{even},\vspace{1mm}\\
           \chi_{11}^{r-1}\cdot w_1,&r\,\,\text{odd}, 
        \end{array}\right.\\
m_r&=\left\{\begin{array}{rl}
           \left[v_1+(-1)^p\omega^{2(jN-2kn)}v_2\right]\boxtimes x_{12}^{2s}x_{21},& r=1, \\
           \chi_{11}^{r-1}\cdot m_1,&r\,\,\text{even},\vspace{1mm}\\
           \chi_{22}^{r-1}\cdot m_1,&r\,\,\text{odd}, 
        \end{array}\right.
\end{align*}
then $\mathscr{I}_{pjk}^s=\bigoplus_{r=1}^{n}\left(\k w_r\oplus\k m_r\right)$ is a $2n$-dimensional Yetter-Drinfeld module over $A_{N\,2n}^{\mu\lambda}$ with the module structure given by 
\begin{align*}
x_{11}\cdot w_r&=\left\{\begin{array}{rl}
                           (-1)^p\omega^{4kn}w_1, & r=1,\\
                           w_{r+1},& r\,\,\text{even}, 1<r< n, \\
                           \omega^{8kn}w_{r-1},& r\,\,\text{odd}, 1<r\leq n, \\
                           (-1)^p\omega^{2n(2k+jN)}w_{n},& r=n\,\,\text{even}, 
                           \end{array}\right.\\
x_{22}\cdot w_r&=\left\{\begin{array}{rl}
                           \omega^{8kn}w_{r-1},& r\,\,\text{even}, 1<r\leq n, \\
                           w_{r+1},& r\,\,\text{odd}, 1\leq r<n, \\
                           (-1)^p\omega^{2n(2k+jN)}w_{n},& r=n\,\,\text{odd}, 
                           \end{array}\right.\\
x_{pq}\cdot w_r&=0, \quad pq=12\,\,\text{or}\,\,21, 1\leq r\leq n, \\
x_{11}\cdot m_r&=\left\{\begin{array}{rl}
                           \omega^{8kn}m_{r-1},& r\,\,\text{even}, 1<r\leq n, \\
                           m_{r+1},& r\,\,\text{odd}, 1\leq r<n, \\
                           \lambda(-1)^p\omega^{2n(2k+jN)}m_{n},& r=n\,\,\text{odd}, 
                           \end{array}\right.\\
x_{pq}\cdot m_r&=0, \quad pq=12\,\,\text{or}\,\,21, 1\leq r\leq n, \\
x_{22}\cdot m_r&=\left\{\begin{array}{rl}
                           (-1)^p\omega^{4kn}m_1, & r=1,\\
                           m_{r+1},& r\,\,\text{even}, 1<r< n, \\
                           \omega^{8kn}m_{r-1},& r\,\,\text{odd}, 1<r\leq n, \\
                           \lambda(-1)^p\omega^{2n(2k+jN)}m_{n},& r=n\,\,\text{even}, 
                           \end{array}\right.
\end{align*}
and the comodule structure given by
\begin{align*}
\rho(w_r)=\left\{\begin{array}{rl}
                   x_{11}^{2(s-r+1)}\chi_{22}^{2r-1}\otimes w_r
                   +x_{12}^{2(s-r+1)}\chi_{21}^{2r-1}\otimes m_r, 
                   & r\,\,\text{even},\vspace{1mm}\\
                 x_{11}^{2(s-r+1)}\chi_{11}^{2r-1}\otimes w_r
                             +x_{12}^{2(s-r+1)}\chi_{12}^{2r-1}\otimes m_r, & r\,\,\text{odd},\\
                 \end{array}\right.
\\
\rho(m_r)=\left\{\begin{array}{rl}
                 x_{11}^{2(s-r+1)}\chi_{11}^{2r-1}\otimes m_r
                 +x_{12}^{2(s-r+1)}\chi_{12}^{2r-1}\otimes w_r, & r\,\,\text{even},\vspace{1mm}\\
                 x_{11}^{2(s-r+1)}\chi_{22}^{2r-1}\otimes m_r
                 +x_{12}^{2(s-r+1)}\chi_{21}^{2r-1}\otimes w_r, & r\,\,\text{odd}.
                 \end{array}\right.
\end{align*}
\item Let  $V_{jk}^\prime=\k v_1^\prime\oplus \k v_2^\prime$, 
$s\in\overline{1,N}$, $p\in\Bbb{Z}_2$, and denote 
\begin{align*}
w_r
&=\left\{\begin{array}{rl} 
v_{1}^\prime\boxtimes \left[x_{11}^{2s}+(-1)^p x_{12}^{2s}\right], &r=1,\\
\chi_{12}^{r-1}\cdot w_{1}, & r\,\text{even and}\, 2\leq r\leq 2n,\vspace{1mm}\\
\chi_{21}^{r-1}\cdot w_{1}, & r\,\text{odd and }\, 1\leq r\leq 2n,
\end{array}\right.
\end{align*}
then $\mathscr{K}_{jk,p}^s=\bigoplus_{r=1}^{2n}w_r$ is a $2n$-dimensional Yetter-Drinfeld module over $A_{N\,2n}^{\mu\lambda}$ with the module structure given by 
\begin{align*}
x_{12}\cdot w_r
&=\left\{\begin{array}{rl}
         w_{r+1}, & r\,\text{odd and }\, 1\leq r<2n,\\
         \bar{\mu}\omega^{8kn} w_{r-1}, & r\,\text{even and }\, 2\leq r\leq 2n,\\
         \end{array}\right.\\
x_{21}\cdot w_r
&=\left\{\begin{array}{rl}
         \lambda\bar{\mu}^{1-n}\omega^{8kn-2n(4kn+jN)}w_{2n}, & r=1,\\
         \bar{\mu}\omega^{8kn} w_{r-1}, & r\,\text{odd and }\, 3\leq r<2n,\\
         w_{r+1}, & r\,\text{even and }\, 2\leq r< 2n,\\
         \lambda\bar{\mu}^{n}\omega^{2n(4kn+jN)}w_1, &r=2n,
         \end{array}\right.\\
x_{pq}\cdot w_r &=0, pq=11\,\text{or}\, 22, 1\leq r\leq 2n,
\end{align*}
and the comodule structure given by 
\begin{align*}
\rho(w_r)=\left\{\begin{array}{rl}
                 x_{11}^{2(s-r+1)}\chi_{11}^{2r-2}\otimes w_r
                 +\Omega_1 
                 x_{12}^{2(s-r+1)}\chi_{12}^{2r-2}\otimes w_{2n-r+2},
                 & 2\leq r\,\text{even},\vspace{1mm}\\
                 x_{11}^{2(s-r+1)}\chi_{22}^{2r-2}\otimes w_r
                 +\Omega_1
                 x_{12}^{2(s-r+1)}\chi_{21}^{2r-2}\otimes w_{2n-r+2}, 
                 &3\leq r\,\text{odd},\vspace{1mm}\\
                 \left[x_{11}^{2s}+(-1)^p x_{12}^{2s}\right]\otimes w_1, &r=1, 
                 \end{array}\right.
\end{align*}
where $\Omega_1=(-1)^p\lambda\bar{\mu}^{r-1-n}\omega^{8kn(r-1-n)-2jNn}$.
\item $\mathscr{M}_{ijk}^s=\bigoplus_{r=1}^{n}\left(\k w_r\oplus\k m_r\right)$, where 
     $s\in\overline{1,N}$, $r\in \overline{1,n}$ and $V_{ijk}=\k v$, 
      \begin{align*}
       w_r=\left\{\begin{array}{rl}
               v\boxtimes x_{11}^{2s+1}, &r=1,\vspace{1mm}\\
               \chi_{22}^{r-1}\cdot w_1, & r\,\,\text{even},\vspace{1mm}\\
               \chi_{11}^{r-1}\cdot w_1, & r\,\,\text{odd},\vspace{1mm}
               \end{array}\right.
       \quad
      m_r=\left\{\begin{array}{rl}
              v\boxtimes x_{12}^{2s}\chi_{21}, &r=1,\vspace{1mm}\\
              \chi_{11}^{r-1}\cdot m_1, & r\,\,\text{even},\vspace{1mm}\\
              \chi_{22}^{r-1}\cdot m_1, & r\,\,\text{odd}.\vspace{1mm}
              \end{array}\right.
      \end{align*}

\item  $\mathscr{N}_{ijk}^s=\bigoplus_{r=1}^{n}\left(\k w_r\oplus\k m_r\right)$, where 
$s\in\overline{1,N}$, $V_{ijk}=\k v$, 
\begin{align*}
w_r=\left\{\begin{array}{rl}
         v\boxtimes x_{12}^{2s+1}, &r=1,\vspace{1mm}\\
         \chi_{22}^{r-1}\cdot w_1, & r\,\,\text{even},\vspace{1mm}\\
         \chi_{11}^{r-1}\cdot w_1, & r\,\,\text{odd},\vspace{1mm}
         \end{array}\right.
\quad
m_r=\left\{\begin{array}{rl}
         v\boxtimes x_{11}^{2s}\chi_{22}, &r=1,\vspace{1mm}\\
         \chi_{11}^{r-1}\cdot m_1, & r\,\,\text{even},\vspace{1mm}\\
         \chi_{22}^{r-1}\cdot m_1, & r\,\,\text{odd}.\vspace{1mm}
         \end{array}\right.
\end{align*} 

\item  $\mathscr{J}_{pjk}^s=\bigoplus_{r=1}^{n}\left(\k w_r\oplus\k m_r\right)$, where 
$V_{jk}=\k v_1\oplus\k v_2$, $s\in\overline{1,N}$, $p\in\Bbb{Z}_2$, 
\begin{align*}
w_r&=\left\{\begin{array}{rl}
           \left[v_1+(-1)^p\omega^{-4kn}v_2\right]\boxtimes x_{12}^{2s+1}, &r=1, \vspace{1mm}\\
           \chi_{22}^{r-1}\cdot w_1,&r\,\,\text{even},\vspace{1mm}\\
           \chi_{11}^{r-1}\cdot w_1,&r\,\,\text{odd}, 
        \end{array}\right.\\
m_r&=\left\{\begin{array}{rl}
           \left[v_1+(-1)^p\omega^{-4kn}v_2\right]\boxtimes x_{11}^{2s}x_{22}, &r=1, \vspace{1mm}\\
           \chi_{11}^{r-1}\cdot m_1,&r\,\,\text{even},\vspace{1mm}\\
           \chi_{22}^{r-1}\cdot m_1,&r\,\,\text{odd}.
        \end{array}\right.
\end{align*}

\item $\mathscr{L}_{jk,p}^s=\bigoplus_{r=1}^{2n}\k w_r$, where 
$V_{jk}^\prime=\k v_1^\prime\oplus \k v_2^\prime$, 
$s\in\overline{1,N}$, $p\in \Bbb{Z}_2$, 
\begin{align*}
w_r
&=\left\{\begin{array}{rl} 
v_{2}^\prime\boxtimes \left[x_{11}^{2s}+(-1)^p x_{12}^{2s}\right], &r=1,\\
\chi_{12}^{r-1}\cdot w_{1}, & r\,\text{even and}\, 2\leq r\leq 2n,\vspace{1mm}\\
\chi_{21}^{r-1}\cdot w_{1}, & r\,\text{odd and }\, 1\leq r\leq 2n.
\end{array}\right.
\end{align*}

\item $\mathscr{Q}_{ijk,p}^s=\bigoplus_{r=1}^{2n}w_r$, where
$V_{ijk}^\prime=\k v$, 
$s\in\overline{1,N}$, $i,j,p\in\Bbb{Z}_2$, $k\in\overline{0,N-1}$,  
\begin{align*}
w_r
&=\left\{\begin{array}{rl} 
v\boxtimes \left[x_{11}^{2s}+(-1)^p x_{12}^{2s}\right], &r=1,\\
\chi_{12}^{r-1}\cdot w_{1}, & r\,\text{even and}\, 2\leq r\leq 2n,\vspace{1mm}\\
\chi_{21}^{r-1}\cdot w_{1}, & r\,\text{odd and }\, 1\leq r\leq 2n.
\end{array}\right.
\end{align*}
\end{enumerate}

\section*{Acknowledgements}
The author thanks  Prof. Nicol{\'a}s 
Andruskiewitsch for constructive comments  and Rongchuan Xiong for helpful discussions.


\begin{thebibliography}{10}

\bibitem{Andruskiewitsch2017}
N.~Andruskiewitsch and I.~Angiono.
\newblock On finite dimensional {N}ichols algebras of diagonal type.
\newblock {\em Bull. Math. Sci.}, 7(3):353--573, 2017.

\bibitem{andruskiewitsch2007pointed}
N.~Andruskiewitsch and F.~Fantino.
\newblock On pointed {H}opf algebras associated with alternating and dihedral
  groups.
\newblock {\em Rev. Un. Mat. Argentina}, 48(3):57--71 (2008), 2007.

\bibitem{Andruskiewitsch2011}
N.~Andruskiewitsch, F.~Fantino, M.~Gra\~{n}a, and L.~Vendramin.
\newblock Finite-dimensional pointed {H}opf algebras with alternating groups
  are trivial.
\newblock {\em Ann. Mat. Pura Appl. (4)}, 190(2):225--245, 2011.

\bibitem{MR3677868}
N.~Andruskiewitsch, C.~Galindo, and M.~M\"uller.
\newblock Examples of finite-dimensional {H}opf algebras with the dual
  {C}hevalley property.
\newblock {\em Publ. Mat.}, 61(2):445--474, 2017.

\bibitem{Andruskiewitsch2018}
N.~Andruskiewitsch and J.~Giraldi.
\newblock Nichols algebras that are quantum planes.
\newblock {\em Linear and Multilinear Algebra}, 66(5):961--991, 2018.

\bibitem{Andruskiewitsch2003MR1994219}
N.~Andruskiewitsch and M.~Gra{\~n}a.
\newblock From racks to pointed {H}opf algebras.
\newblock {\em Adv. Math.}, 178(2):177--243, 2003.

\bibitem{MR1659895}
N.~Andruskiewitsch and H.-J. Schneider.
\newblock Lifting of quantum linear spaces and pointed {H}opf algebras of order
  {$p^3$}.
\newblock {\em J. Algebra}, 209(2):658--691, 1998.

\bibitem{Burciu2017}
S.~Burciu.
\newblock On the {G}rothendieck rings of generalized {D}rinfeld doubles.
\newblock {\em J. Algebra}, 486:14--35, 2017.

\bibitem{MR2037722}
C.~C{\u{a}}linescu, S.~D{\u{a}}sc{\u{a}}lescu, A.~Masuoka, and C.~Menini.
\newblock Quantum lines over non-cocommutative cosemisimple {H}opf algebras.
\newblock {\em J. Algebra}, 273(2):753--779, 2004.

\bibitem{Fantino2019}
F.~Fantino, G.~Garc\'{\i}a, and M.~Mastnak.
\newblock On finite-dimensional copointed {H}opf algebras over dihedral groups.
\newblock {\em J. Pure Appl. Algebra}, 223(8):3611--3634, 2019.

\bibitem{gould1993quantum}
M.~D. Gould.
\newblock Quantum double finite group algebras and their representations.
\newblock {\em Bull. Austral. Math. Soc.}, 48(02):275--301, 1993.

\bibitem{MR2207786}
I.~Heckenberger.
\newblock The {W}eyl groupoid of a {N}ichols algebra of diagonal type.
\newblock {\em Invent. Math.}, 164(1):175--188, 2006.

\bibitem{heckenberger2009classification}
I.~Heckenberger.
\newblock Classification of arithmetic root systems.
\newblock {\em Adv. Math.}, 220(1):59--124, 2009.

\bibitem{Heckenberger[2020]copyright2020}
I.~Heckenberger and H.-J. Schneider.
\newblock {\em Hopf algebras and root systems}, volume 247 of {\em Mathematical
  Surveys and Monographs}.
\newblock American Mathematical Society, Providence, RI, [2020] \copyright
  2020.

\bibitem{MR3605018}
I.~Heckenberger and L.~Vendramin.
\newblock A classification of {N}ichols algebras of semisimple
  {Y}etter-{D}rinfeld modules over non-abelian groups.
\newblock {\em J. Eur. Math. Soc. (JEMS)}, 19(2):299--356, 2017.

\bibitem{Heckenberger2017}
I.~Heckenberger and L.~Vendramin.
\newblock The classification of {N}ichols algebras over groups with finite root
  system of rank two.
\newblock {\em J. Eur. Math. Soc. (JEMS)}, 19(7):1977--2017, 2017.

\bibitem{MR2336009}
J.~Hu and Y.-H. Zhang.
\newblock The {$\beta$}-character algebra and a commuting pair in {H}opf
  algebras.
\newblock {\em Algebr. Represent. Theory}, 10(5):497--516, 2007.

\bibitem{MR2352888}
J.~Hu and Y.-H. Zhang.
\newblock Induced modules of semisimple {H}opf algebras.
\newblock {\em Algebra Colloq.}, 14(4):571--584, 2007.

\bibitem{Liu2019}
Z.-M. Liu and S.-L. Zhu.
\newblock On the structure of irreducible {Y}etter-{D}rinfeld modules over
  quasi-triangular {H}opf algebras.
\newblock {\em J. Algebra}, 539:339--365, 2019.

\bibitem{Majid1991}
S.~Majid.
\newblock Doubles of quasitriangular {H}opf algebras.
\newblock {\em Comm. Algebra}, 19(11):3061--3073, 1991.

\bibitem{MR1800713}
A.~Masuoka.
\newblock Cocycle deformations and {G}alois objects for some cosemisimple
  {H}opf algebras of finite dimension.
\newblock In {\em New trends in {H}opf algebra theory ({L}a {F}alda, 1999)},
  volume 267 of {\em Contemp. Math.}, pages 195--214. Amer. Math. Soc.,
  Providence, RI, 2000.

\bibitem{Milinski2000}
A.~Milinski and H.-J. Schneider.
\newblock Pointed indecomposable {H}opf algebras over {C}oxeter groups.
\newblock In {\em New trends in {H}opf algebra theory ({L}a {F}alda, 1999)},
  volume 267 of {\em Contemp. Math.}, pages 215--236. Amer. Math. Soc.,
  Providence, RI, 2000.

\bibitem{MR2386730}
D.~Naidu and D.~Nikshych.
\newblock Lagrangian subcategories and braided tensor equivalences of twisted
  quantum doubles of finite groups.
\newblock {\em Comm. Math. Phys.}, 279(3):845--872, 2008.

\bibitem{Natale2003}
S.~Natale.
\newblock On group theoretical {H}opf algebras and exact factorizations of
  finite groups.
\newblock {\em J. Algebra}, 270(1):199--211, 2003.

\bibitem{radford2003oriented}
D.~E. Radford.
\newblock On oriented quantum algebras derived from representations of the
  quantum double of a finite-dimensional {H}opf algebra.
\newblock {\em J. Algebra}, 270(2):670--695, 2003.

\bibitem{Schauenburg1992}
P.~Schauenburg.
\newblock {\em On coquasitriangular {H}opf algebras and the quantum
  {Y}ang-{B}axter equation}, volume~67 of {\em Algebra Berichte [Algebra
  Reports]}.
\newblock Verlag Reinhard Fischer, Munich, 1992.

\bibitem{Shi2019}
Y.-X. Shi.
\newblock Finite-dimensional {H}opf algebras over the {K}ac-{P}aljutkin algebra
  {$H_8$}.
\newblock {\em Rev. Un. Mat. Argentina}, 60(1):265--298, 2019.

\bibitem{Shi2020odd}
Y.-X. Shi.
\newblock Finite dimensional {N}ichols algebras over {S}uzuki algebra
  \uppercase\expandafter{\romannumeral2}: over simple {Y}etter-{D}rinfeld
  modules of ${A}_{N\,2n+1}^{\mu\lambda}$.
\newblock {\em arXiv:2103.06475}, 2021.

\bibitem{Suzuki1998}
S.~Suzuki.
\newblock A family of braided cosemisimple {H}opf algebras of finite dimension.
\newblock {\em Tsukuba J. Math.}, 22(1):1--29, 1998.

\bibitem{Takeuchi2000}
M.~Takeuchi.
\newblock Survey of braided {H}opf algebras.
\newblock In {\em New trends in {H}opf algebra theory ({L}a {F}alda, 1999)},
  volume 267 of {\em Contemp. Math.}, pages 301--323. Amer. Math. Soc.,
  Providence, RI, 2000.

\bibitem{Takeuchi2002}
M.~Takeuchi.
\newblock A short course on quantum matrices.
\newblock In {\em New directions in {H}opf algebras}, volume~43 of {\em Math.
  Sci. Res. Inst. Publ.}, pages 383--435. Cambridge Univ. Press, Cambridge,
  2002.
\newblock Notes taken by Bernd Str\"{u}ber.

\bibitem{Wakui2003}
M.~Wakui.
\newblock The coribbon structures of some finite dimensional braided {H}opf
  algebras generated by {$2\times 2$}-matrix coalgebras.
\newblock In {\em Noncommutative geometry and quantum groups ({W}arsaw, 2001)},
  volume~61 of {\em Banach Center Publ.}, pages 333--344. Polish Acad. Sci.
  Inst. Math., Warsaw, 2003.

\bibitem{Wakui2010a}
M.~Wakui.
\newblock Polynomial invariants for a semisimple and cosemisimple {H}opf
  algebra of finite dimension.
\newblock {\em J. Pure Appl. Algebra}, 214(6):701--728, 2010.

\bibitem{Wakui2019}
M.~Wakui.
\newblock Braided {M}orita equivalence for finite-dimensional semisimple and
  cosemisimple {H}opf algebras.
\newblock In {\em Proceedings of the {M}eeting for {S}tudy of {N}umber
  {T}heory, {H}opf {A}lgebras and {R}elated {T}opics}, pages 157--183. Yokohama
  Publ., Yokohama, 2019.

\end{thebibliography}
\end{document}